\documentclass[12pt]{amsart}
\usepackage{amsmath,amsfonts,amssymb,amsxtra,amscd,enumerate,amsthm}
\usepackage{color}
\usepackage{latexsym}
\usepackage{epsfig}
 \usepackage{fullpage}

\newcommand\les{\lesssim}
\newcommand\ges{\gtrsim}

\newcommand\bj{\textbf{j}}
\newcommand{\bx}{\mathbf{x}}

\newcommand\R{\mathbb{R}}

\newcommand\C{\mathbb{C}}
\newcommand\Z{\mathbb{Z}}
\newcommand\N{\mathbb{N}}

\newcommand{\calQ}{\mathcal Q}
\newcommand{\calT}{\mathcal T}
\newcommand{\calE}{\mathcal E}
\newcommand{\calH}{\mathcal H}
\newcommand{\calC}{\mathcal C}

\newcommand{\calR}{\mathcal R}
\newcommand{\calN}{\mathcal N}

\newcommand{\ls}{{\lesssim}}

\newcommand\la{\langle}
\newcommand\ra{\rangle}
\newtheorem{theo}{Theorem}
\numberwithin{theo}{section} 
\newtheorem{lema}[theo]{Lemma}
\newtheorem{prop}[theo]{Proposition}
\newtheorem{corol}[theo]{Corollary}
\newtheorem{defin}[theo]{Definition}
\newtheorem{conjecture}[theo]{Conjecture}

\newtheorem{rem}{Remark}

\newtheorem{set}{Setup}

\numberwithin{equation}{section}

\begin{document}
\title{Optimal multilinear restriction estimates for a class of surfaces with curvature}

\author[I. Bejenaru]{Ioan Bejenaru} \address{Department
  of Mathematics, University of California, San Diego, La Jolla, CA
  92093-0112 USA} \email{ibejenaru@math.ucsd.edu}

\begin{abstract} In \cite{BeCaTa}, Bennett, Carbery and Tao consider the $k$-linear restriction estimate in $\R^{n+1}$ and establish the
near optimal $L^\frac2{k-1}$ estimate under transversality assumptions only.
In \cite{Be3} we have shown that the trilinear restriction estimate improves its range of exponents under some curvature assumptions.  
In this paper we establish almost sharp multilinear estimates for a class of hypersurfaces with curvature for $4 \leq k \leq n$. Together with previous results in the literature, this shows that curvature improves the range of exponents in the multilinear restriction estimate at all levels of lower multilinearity, that is when $k \leq n$.

\end{abstract}


\maketitle

\section{Introduction} \label{intr}

For $n \geq 1$, let $U \subset \R^{n}$ be an open, bounded and connected  neighborhood of the origin and let $\Sigma: U \rightarrow \R^{n+1}$ be a smooth parametrization of an $n$-dimensional submanifold of $\R^{n+1}$ (hypersurface), which we denote by $S=\Sigma(U)$. To this parametrization of $S$ we associate the operator $\calE$ defined 
by
\[
\calE f(x) = \int_U e^{i x \cdot \Sigma(\xi)} f(\xi) d\xi. 
\]
Given $k$ smooth, compact hypersurfaces $S_i \subset \R^{n+1}, i=1,..,k$, where $1 \leq k \leq n+1$, the $k$-linear restriction estimate is the following
inequality 
\begin{equation}  \label{MRE}
\| \Pi_{i=1}^k \calE_i f_i \|_{L^p(\R^{n+1})} \les \Pi_{i=1}^k \| f_i \|_{L^2(U_i)}.  
\end{equation}
In a more compact format this estimate is abbreviated as follows:
\[
\calR^*(2 \times ... \times 2 \rightarrow p).
\] 
The fundamental question regarding the above estimate is the value of the optimal $p$ for which it holds true. Given that the 
estimate $\calR^*(2 \times ... \times 2 \rightarrow \infty)$ is trivial, the optimality is translated into the smallest $p$ for which the
estimate holds true. In  \cite{BeCaTa} Bennett, Carbery and Tao clarified the role of transversality between the surfaces involved and established that,
under a transversality condition between $S_1,..,S_k$, the optimal exponent is $p=\frac2{k-1}$; the actual result in \cite{BeCaTa} is near-optimal, and the optimal problem is currently open. The optimality can be easily revealed by taking $S_i$ to be transversal hyperplanes, in which case the estimate becomes the classical Loomis-Whitney inequality.  

It is also known, in some cases (precisely when $k \leq 2$), or expected, in most of the others, that curvature assumptions improve the range of exponents in \eqref{MRE}, except for the case $k=n+1$. In \cite{Be3} we formalized the following
\begin{conjecture} \label{conjj}
Under appropriate transversality and curvature conditions on the surfaces $S_i$, $\calR^*(2 \times ... \times 2 \rightarrow p)$ 
holds true for any $p \geq p(k)=\frac{2(n+1+k)}{k(n+k-1)}$. 
\end{conjecture}

The case $k=1$ has been understood for a very long time. Without any curvature assumptions, the optimal exponent is $p=\infty$; once the surface has some non-vanishing principal curvatures, the exponent improves to $p=\frac{2(l+2)}{l} $, where $l$ is the number of non-vanishing principal curvatures. The case of non-zero Gaussian curvature, corresponding to $l=n$, is the classical result due to Tomas-Stein, see \cite{St}.

The case $k=2$ without any curvature assumptions corresponds to the classical $L^2$ bilinear estimate, where the optimal estimate had been established. Once curvature curvature assumptions are allowed, the best possible exponent in 
$\calR^*(2 \times 2 \rightarrow p)$ is $p=\frac{n+3}{n+1}$ and it was conjectured in \cite{FoKl}. The problem was intensely studied, see
\cite{Bou-CM, Wo, Tao-BW, Tao-BP, TV-CM1, Lee-BR,LeeVa, Be2} and references therein. The problem is solved in the regime $p > \frac{n+3}{n+1}$
for general hypersurfaces with curvature; the end-point $p = \frac{n+3}{n+1}$ is solved only for the cones, see Tao \cite{Tao-BW}. 

The case $k=n+1$ is fairly well-understood. We note that in this case, additional curvature assumptions have no effect 
on the optimality of $p$. It is conjectured that if the hypersurfaces $S_i \subset \R^{n+1}$ are transversal, then \eqref{MRE} holds true for $p \geq p_0=\frac{2}n$. If $S_i$ are transversal  hyperplanes, \eqref{MRE}  is the classical Loomis-Whitney inequality and its proof is elementary.  Once the surfaces are allowed to have non-zero principal curvatures, things become far more complicated and the problem has been the subject of extensive research, see \cite{BeCaTa,Gu-main} and references therein. 
In \cite{BeCaTa}, Bennett, Carbery and Tao establish a near-optimal version of \eqref{MRE}: this is \eqref{MRE} 
with an additional $R^\epsilon$ factor when the estimate is made over balls of radius $R$ in $\R^{n+1}$. The optimal result for \eqref{MRE}, that is without the $\epsilon$-loss, is an open problem; in some cases one can use $\epsilon$-removal techniques to derive the result without the $\epsilon$-loss for $p > \frac2n$, see \cite{BoGu} for the case of surfaces with non-vanishing Gaussian curvature. The end-point for the multilinear Kakeya version of \eqref{MRE} (a slightly weaker statement than \eqref{MRE}) has been established by Guth in \cite{Gu-main} using tools from algebraic topology. 

In the remaining cases, $3 \leq k \leq n$,  the $k$-linear restriction theory has been addressed in \cite{BeCaTa} only under 
transversality assumptions and the authors established the near-optimal result for $p \geq \frac{2}{k-1}$. The exponent $\frac2{k-1}$ is sharp for generic surfaces, but it is not the optimal exponent once curvature assumptions are brought into the problem; indeed note that $p(k) < \frac2{k-1}$.  

In \cite{Be3} we looked at the trilinear estimate (corresponding to $k=3$) and proved the Conjecture \ref{conjj} in the regime $p > p(3)$ for a particular class of surfaces: the double-conic ones. These surfaces have the nice property that they have the exact "amount" of curvature to obtain the estimate with the optimal exponent $p(3)$, and no more, in the sense that they are "flat" in the unnecessary directions. 

In this paper we provide the equivalent result for $4 \leq k \leq  n$ for $k-1$-conical surfaces. We note that passing from the case $k=3$
to $k\geq 4$ requires not only additional technical ideas, but also conceptual ones. 

We describe bellow  the class of hypersurfaces for which we prove the Conjecture \ref{conjj}.  We start with the definition of a foliation. A $k-1$-dimensional foliation of the ($n$-dimensional) hypersurface $S$ is a decomposition of $S$ into a union of connected disjoint sets $\{ S_\alpha \}_{\alpha \in A}$, called the leaves of the foliation, with the following property: every point in $S$ has a  neighborhood $V$ and local system of coordinates $x: V \subset S \rightarrow \R^{n}$ such that for each leaf $S_{\alpha}$, the coordinates of $V \cap S_{\alpha}$ are $\xi_k=constant,..,\xi_{n}=constant$. 

We now formalize the conditions we impose on our surfaces. As before, $S_i, i \in \{1,..,k \}$ are hypersurfaces with smooth parameterizations $\Sigma_i: U_i \subset \R^n \rightarrow \R^{n+1}$, where each $U_i$ is open, bounded and connected  neighborhood of the origin (note that different $U_i$ may belong to different hyperplanes identified with the same $\R^n$). In addition, we assume the following three hypothesis:

i) (foliation) for each $i \in \{1,..,k\}$, the hypersurface $S_i$ admits the foliation
\[
S_i = \bigcup_{\alpha} S_{i,\alpha}
\]
where, for each $\alpha$, the leaf $S_{i,\alpha}$ is a flat submanifold of dimension $k-1$.  
 
ii) (the leaves are completely flat) If $S_{N_i(\zeta_i)}$ is the shape operator of $S_i$ at $\zeta_i \in S_i$ with choice of normal $N_i(\zeta_i)$, we assume that for every $v \in T_{\zeta_i} S_{i,\alpha}$ (the tangent plane at $S_{i,\alpha}$ at the point $\zeta_i \in S_{i,\alpha}$) the following holds
\[
S_{N_i(\zeta_i)} v =0. 
\]

iii) (transversality and curvature) There exists $\nu >0$ such that for any $\zeta_i \in S_i, i \in \{1,..,k\}$, for any $l \in \{1,..,k\}$ and for any orthonormal basis 
$v_{k+1},..,v_{n+1} \in (T_{\zeta_l} S_{l,\alpha})^\perp \subset T_{\zeta_l} S_{l,\alpha}$ the following holds true 
\begin{equation} \label{curva}
vol ( N_1(\zeta_1), .., N_k(\zeta_k), S_{N_l(\zeta_l)} v_{k+1} ,  ... , S_{N_l(\zeta_l)} v_{n+1} ) \geq \nu. 
\end{equation}

In \eqref{curva} $vol$ is the standard volume form of $n+1$ vectors in $\R^{n+1}$, thus the condition quantifies the linear independence of the vectors $N_1(\zeta_1), .., N_k(\zeta_k), S_{N_l(\zeta_l)} v_{k+1} ,  ... , S_{N_l(\zeta_l)} v_{n+1}$. 

The condition ii) says that $S_{i,\alpha}$ are, in some sense, completely flat components of $S_i$ since, besides being subsets of
affine planes of dimension $k-1$, the normal $N(\zeta)$ to $S_i$ is constant as we vary $\zeta$ along $S_{i,\alpha}$ for fixed $\alpha$. 

The first things to read in condition iii) is the transversality condition between $S_1, .., S_k$ due to the transversality between any choice on normals.  The condition iii) also says that the submanifolds transversal to the leafs carry the curvature assumptions, in the sense that their tangent space does not contain any eigenvectors of the shape operator. In addition, for each $i \in \{1,..,k\}$, we are guaranteed to have transversality between $N_1(\zeta_1), .. , N_k(\zeta_k)$ and $S_{N_i} ( T_{\zeta_i} (S_{l,\alpha})^\perp) $. 
 
 In fact iii) is equivalent to the apparently weaker condition:
 
 iii') There exists $\nu >0$ such that for any $\zeta_i \in S_i, i \in \{1,..,k\}$, for any $l \in \{1,..,k\}$ and for any  unit vector
 $v \in (T_{\zeta_l} S_{l,\alpha})^\perp \subset T_{\zeta_l} S_{l,\alpha}$ the following holds true 
\begin{equation} \label{curva2}
vol ( N_1(\zeta_1), .., N_k(\zeta_k), S_{N_l(\zeta_l)} v) \geq \nu. 
\end{equation}
 Obviously here $vol$ stands for the $k+1$-dimensional volume of the vectors. 
 
At this point we can state the main result of this paper. 
\begin{theo} \label{mainT}  Assume that $S_1,..,S_k$ satisfy the conditions i)-iii) above. 
Given any $p$ with $p(k)=\frac{2(n+k+1)}{k(n+k-1)}< p \leq \infty$, the following holds true
\begin{equation} \label{mainE}
\| \Pi_{i=1}^k \calE_i f_i \|_{L^p(\R^{n+1})} \leq C(p) \Pi_{i=1}^k \| f_i \|_{L^2(U_i)}, \quad \forall f_i \in L^2(U_i).   
\end{equation}
\end{theo}

To our best knowledge this result is the first instance when the $k$-linear restriction estimate, with $4 \leq k \leq n$,
is proved for the almost optimal exponent, that is $p > p(k)$. However, very recently Guth formulated in  \cite{Gu-II} a weaker version 
of Conjecture \ref{conjj} which he proves in the case when $S_i$ are subsets of the paraboloid, and  for the same range of parameters $p(k)< p \leq \infty$. The formulation of this weaker version is technical and we skip it here. In \cite{Gu-II}, Guth uses this weaker version to improve the ranges of the linear restriction theory. It is important to note that Guth employs polynomial partition methods to prove his result.
The arguments we use in this paper are very different, see details below.

The result in Theorem \ref{mainT} and the corresponding one in \cite{Be3} show that the Conjecture \ref{conjj} holds true at
least in some model cases. We hope that this result would lead the way towards a complete resolution of the Conjecture,  
which, in turn, should have important consequences. The multilinear theory discussed above has had major impact in other problems. We mention a few such examples: In Harmonic Analysis, the bilinear and 
$n+1$ restriction theory was used to improve results in the context of Schr\"odinger maximal function, see \cite{Bou-SMF,Lee-SMF,TV-CM2, DL},  restriction conjecture, see \cite{Tao-BP,BoGu, Gu-I, Gu-II}, the decoupling conjecture, see \cite{BoDe,BoDeGu}. In Partial Differential Equations, the linear theory inspired the Strichartz estimates, see \cite{Tao-book}, while the bilinear restriction theory is used in the context of more sophisticated techniques, such as the profile decomposition, see \cite{MeVe}, and concentration compactness methods, see \cite{KeMe}. 

Theorem \ref{mainT} reveals the following geometric feature: the optimal $k$-linear restriction estimate discards the effect of $k-1$ curvatures; indeed, each $S_i$ has precisely $k-1$ vanishing principal curvatures, thus it relies only on $n+1-k$ principal curvatures being non-zero, although the actual statements have to be more rigorous. This geometric feature of the problem was conjectured by Bennett, Carbery and Tao in \cite{BeCaTa}.

We continue with  an overview of the paper and highlight some of the elements used the proof of Theorem \ref{mainT}.  The reader may look at the paper as split into two parts: Sections \ref{GS} through \ref{SI} and Sections \ref{LocME} and \ref{CMS}, with Section \ref{SePa} marking the transition between the two. In Sections \ref{GS} through \ref{SI} we adapt to our current setup the standard arguments that are similar to our previous works in the bilinear and trilinear setup, see \cite{Be2} and \cite{Be3}: overview of the geometry of the problem, wave packet theory, table construction and the induction on scales argument. All these ideas originate from the work of Tao \cite{Tao-BW}.

The second part of the paper, Sections \ref{LocME} and \ref{CMS}, contains the novel ideas in this paper and they play a key role in establishing
the improved estimate \eqref{Phia2} in Section \ref{SI}. We note that the equivalent results (to those in Sections \ref{LocME} and \ref{CMS} here) 
in the bilinear and trilinear theory are much simpler, given the structure of the problem, and can be easily derived inside the body of the main argument. 
The results in Sections \ref{LocME} and \ref{CMS} hold in the context of general hypersurfaces, in particular they do not assume the foliation structure or curvature properties used in Theorem \ref{mainT}. We also think that these results are new in the literature and may be of independent interest.

The starting ideas originate in the prior work of the author on the multilinear restriction estimate in \cite{Be1}. 
In that paper we proved that the $k$-linear restriction estimate
\begin{equation} \label{tri2}
\| \Pi_{i=1}^k \calE_i  f_i  \|_{L^\frac2{k-1}(B(0,r))} \leq C(\epsilon) r^\epsilon 
\Pi_{i=1}^k \| f_{i} \|_{L^2(U_{i})}.
\end{equation}
improves under appropriate localizations of one the factors $f_i$. These localizations are precisely the ones carried by the wave packets
appearing in the decomposition of one of the factors $\calE_i f_i$, and one needs to obtain an appropriate estimate for such superpositions 
of wave packets. This was an easy task in the case of the trilinear estimate: the estimate is in $L^1$ and the triangle inequality 
holds true. The triangle inequality fails to hold true in the spaces $L^\frac2{k-1}$ with $k \geq 4$; the way to deal with this aspect is to further refine the techniques developed in \cite{Be1} and derive good "off-diagonal" type estimates, which in turn give the desired superposition estimate with the correct localization gain, see Theorem \ref{MB2}. 
A further localization to cubes is needed for technical reasons, see Corollary \ref{MBcor}. This analysis is carried out in Section \ref{LocME}.

In Section \ref{CMS} we prove the following estimate:
\begin{equation} \label{CKi}
\|  \| \calE_1 f_1 \|_{L^2(S(q))} \Pi_{i=2}^{k'} \| \calE_i f_i \|_{L^2(q)}\|_{l^\frac2{k'-1}_q}
 \leq C(\epsilon) r^\frac{k'}2 r^\epsilon \Pi_{i=1}^{k} \| f_i \|_{L^2(U_i)}.
\end{equation}
We note that the $k'$ above is related to the $k$ in our problem by the simple relation $k'=k-1$. Here $q$ are cubes of size $r$
and the $l^\frac{2}{k'-1}_q$ is taken over such cubes contained in a larger cube of size $r^2$; $S(q)=S+q$ where $S$ is a surface with
some "good" properties. This estimate has the character of a $k$-linear restriction estimate, although it is more complex due to the factor 
$ \| \calE_1 f_1 \|_{L^2(S(q))}$; if $S$ was a point (that is of dimension zero), then the above estimate is similar to the $k$-linear restriction estimate; however the surface we encounter has the maximal dimension that allows \eqref{CKi} to hold true. Another interesting aspect is that 
the maximal dimension of $S$ saturates the estimate \eqref{CKi} in the following sense: while for $k < n+1$, \eqref{tri2} improves under appropriate localizations of some $f_i$, \eqref{CKi} does not, just as the $n+1$-linear restriction estimate does not improve under localizations. 

We identified \eqref{CKi} as the necessary ingredient to closing the improved estimate \eqref{Phia2} in Section \ref{SI}.
We note that in the bilinear theory $k'=1$, and the estimate \eqref{CKi} corresponds to an energy estimate for a free wave across hypersurfaces 
that are transversal to its directions of propagation; this is a classical tool in PDE. In the trilinear theory $k'=2$ the estimate \eqref{CKi} 
is an $l^2$ type one that can be dealt with in a direct manner, by using wave packet decompositions for both free waves and some analysis on their interaction. 
It is in the case $k \geq 4$, that the true character of \eqref{CKi} comes into light. The analysis of the estimate \eqref{CKi} is carried out in Section \ref{CMS}.

\subsection{Notation}

We start by clarifying the role of various constants that appear in the argument. 
$N$ is a large integer that depends only on the dimension. $C$ is a large constant that may change
from line to line, may depend on $N$, but not on $c$ and $C_0$ introduced below. $C$ is used in the definition 
of: $A \ls B$, meaning $A \leq C B$,  $A \ll B$, meaning $A \leq C^{-1} B$, and $A \approx B$, meaning $A \les B \wedge B \les A$.
For a given number $r \geq 0$, by $A = O(r)$ we mean that $A \approx r$.  $C_0$ is a constant that is independent of any other constant and its role is to reduce the size of cubes in the inductive argument.  It can be set $C_0=4$ throughout the argument, but we keep it this way so that its role in the argument is not lost. $c \ll 1$ is a very small variable meant to make expressions $\ll 1$ and most estimates will be stated to hold in a range of $c$.

We use the standard notation $(\xi_1,..,\bar \xi_i,..\xi_{l}):= (\xi_1,..,\xi_{i-1}, \xi_{i+1}..,\xi_{l})$ to mean that one component is missing.

By powers of type $R^{\alpha+}$ we mean $R^{\alpha+\epsilon}$ for arbitrary $\epsilon > 0$. Practically they should be seen
as $R^{\alpha+\epsilon}$ for arbitrary $0 < \epsilon \ls 1$. The estimates where such powers occur will obviously depend on $\epsilon$. 

By $B(x,R)$ we denote the ball centered at $x$ with radius $R$ in the underlying space (most of the time it will be $\R^n$ or $\R^{n+1}$). 

Let $\eta_0:\R^n \rightarrow [0,+\infty)$ be a Schwartz function, normalized in $L^1$, that is $\| \eta_0 \|_{L^1}=1$,
and with Fourier transform supported on the unit ball. Given some $r>0$ we denote by $\eta_r(x)=r^{-n} \eta_0(r^{-1} x)$
and note that $\hat \eta_r$ is supported in $B(0,r)$. We will abuse notation and use the same $\eta_0$ for functions with 
the same properties, but with a different base space, such as $\eta_0:\R^{n+1} \rightarrow [0,+\infty)$.

A disk $D \subset \R^{n+1}$ has the form
\[
D=D(x_D,t_D;r_D)=\{(x,t_D) \in \R^{n+1}: |x-x_D| \leq r_D\},
\]
for some $(x_D,t_D) \in \R^{n+1}$ and $r_D > 0$. We define the associated smooth cut-off
\[
\tilde \chi_{D}(x,t)= (1+\frac{|x-x_D|}{r_D})^{-N}. 
\] 

A cube $Q \subset \R^{n+1}$ of size $R$ has the standard definition $Q=\{(x,t) \in \R^{n+1}: \|(x-x_Q,t-t_Q) \|_{l^\infty} \leq \frac{R}2 \}$, where $c_Q=(x_Q,t_Q)$ is the center of the cube. Given a constant $\alpha >0$ we define $\alpha Q$ to be the dilated by $\alpha$ of $Q$ from its center, that is
$\alpha Q=\{(x,t) \in \R^{n+1}: \|(x-x_Q,t-t_Q) \|_{l^\infty} \leq \alpha \cdot \frac{R}2 \}$. 

Given a cube $q \subset \R^{n+1}$ of size $r$
we will use two functions that are highly concentrated in $q$. One is build with the help of $\eta_0$ (as mentioned earlier, we abuse
notation here as we should be using the corresponding $\eta_0:\R^{n+1} \rightarrow [0,+\infty)$ with similar properties):
\[
\chi_{q}(x) = \eta_0 (\frac{x-c(q)}r).
\]
This localization function has nice properties on the Fourier side. The other localization function is
\[
\tilde \chi_{q}(x) = (1+ |\frac{x-c(q)}r|)^{-N},
\]
for some large $N$. This localization has better properties on the physical side. 

We recall the standard estimate for superpositions of functions in $L^p$ for $p \leq 1$:
\begin{equation} \label{triangle}
\| \sum_\alpha f_\alpha \|_{L^p}^p \leq \sum_\alpha \| f_\alpha \|_{L^p}^p,
\end{equation}
as well as the following estimate for sequences 
\begin{equation} \label{ab}
\| a_i \cdot b_i \|_{l^{\frac2{k-1}}_i} \les \| a_i \|_{l^2_i}  \| b_i \|_{l^{\frac2{k-2}}_i}. 
\end{equation}

\subsection*{Acknowledgement}
Part of this work was supported by a grant from the Simons Foundation ($\# 359929$, Ioan Bejenaru).
Part of this work was supported by the National Science Foundation under grant No. $0932078000$ while the author was in residence at the 
Mathematical Research Sciences Institute in Berkeley, California, during the Fall 2015 semester.

\section{Geometry of the surfaces and consequences} \label{GS}

We start by simplifying the setup. The surfaces are bounded, therefore we can always break them into smaller (and similar) pieces such as to accommodate the additional hypothesis described bellow. 

First we note that we can assume each $S_i$ to be of graph type: there is a smooth map 
$\varphi_i: U_i \subset \R^n \rightarrow \R$ such that $S=\{\Sigma_i(\xi)= (\xi,\varphi_i(\xi)): \xi \in U_i \}$.
Here $U_i$ are open, connected with compact closure. 
It is less important that the graphs are of type $\zeta_{n+1}=\varphi_i(\zeta_1,..,\zeta_{n})$ (we can have as well $\zeta_k=\varphi_i(\zeta_1,..,\bar \zeta_k,..\zeta_{n+1})$), although we can accommodate this by a rotation of coordinates. Then each flat leaf $S_{i,\alpha}$ corresponds to a flat leaf $U_{i,\alpha}$, in the sense that $\Sigma_i(U_{i,\alpha})=S_{i,\alpha}$; this is indeed the case since projections onto hyperplanes along a vector transversal to $S_i$ takes $k-1$-dimensional affine planes to $k-1$-dimensional affine planes. 

 We can find a system of coordinates $\bx_i: \R^n \rightarrow \R^n$ that parametrizes each leaf $U_{i,\alpha}$ into a new flat leaf 
$\tilde U_{i,\alpha}$ characterized by $\xi_k=constant,..,\xi_{n}=constant$. Finally, we assume that each $U_i$ has small enough diameter. 

Next, we derive a key geometric consequence of our setup. Given a surface $S_i$ we define $\calN_i=\{ N_i(\zeta_i): \zeta_i \in S_i \}$ to 
be the set of normals at $S_i$. By $dspan \calN_i$ we denote the following subset of the classical span of $\calN_i$:
\[
dspan \calN_i = \{ \alpha N_\alpha + \beta N_\beta: N_\alpha, N_\beta \in \calN_i, \alpha, \beta \in \R \}.
\] 
$dspan \calN_i $ is the set of linear combinations of two vectors in $\calN_i$; it is not a linear subspace. 

Given a set of indexes $I \subset \{1,2,..,k\}$ we also define 
\[
d \calN_I = \{ \alpha N_\alpha + \beta N_\beta: N_\alpha \in \calN_i, N_\beta \in \calN_j, i,j \in I, i \neq j, \alpha, \beta \in \R \}.
\]

With these notation in place, we claim the following result.
\begin{lema} \label{GL}
Assume $S_i, i=1,..,k$ satisfy the conditions i)-iii). Let $I=\{3,..,k\}$. Then 
for any $N \in dspan \calN_1$, $N_2 \in \calN_2$ and $\tilde N \in d\calN_{I}$,
the following holds true:
\begin{equation} \label{GLe}
vol( N, N_2, \tilde N| \ges |N| \cdot |N_2| \cdot |\tilde N| .
\end{equation}
 \end{lema}
Obviously the above statement is symmetric as we can switch the particular role of each $S_i, i=1,..,k$ plays in the above estimate.

\begin{proof} The proof is similar to the one provided in \cite{Be3}. We write $N=\alpha N_\alpha + \beta N_\beta$ for some $N_\alpha \ne N_\beta$ and consider $\gamma: [0,t_0] \rightarrow S_1$, a smooth curve with the property that 
$ N_1(\gamma(0))=N_\alpha$ and $ N_1(\gamma(t_0))=N_\beta$. We also assume that $|\gamma'(t)| = 1$ on $[0,t_0]$ and that $0 \leq t_0 \ll 1$; this is possible because we assumed $U_1$ of small diameter.  In addition, if $\alpha_0$ is such that $\gamma(0) \in S_{i,\alpha_0}$, we can assume that 
$\gamma'(0) \in (T_{\gamma(0)} S_{1,\alpha_0})^\perp$.  Then we have
\[
\begin{split}
N_1(\gamma(t_0)) & = N_1(\gamma(0)) + \int_{0}^{t_0} S_{N_1(\gamma(s))} \gamma'(s) ds \\
& = N_1(\gamma(0)) + t_0 S_{N_1(\gamma(0))} \gamma'(0) + O(t_0^2).  
\end{split}
\]
We then continue with
\[
\begin{split}
N & = \alpha N_1(\gamma(0))+ \beta (N_1(\gamma(0)) + t_0 S_{N_1(\gamma(0))} \gamma'(0) + O(t_0^2))  \\
& = (\alpha+\beta) N_1(\gamma(0)) + \beta t_0 S_{N_1(\gamma(0))} \gamma'(0) + \beta O(t_0^2) 
\end{split}
\]
The two vectors $N_1(\gamma(0))$ and $S_{N_1(\gamma(0))} \gamma'(0)$ are transversal,  thus 
$|N| \approx |\alpha+\beta| + t_0 |\beta|  |S_{N_1(\gamma(0))} \gamma'(0)|$ (here we use that $t_0 \ll 1$), and also 
\[
\begin{split}
vol(N,N_2, \tilde N) & \approx vol((\alpha+\beta) N_1(\gamma(0)) + \beta t_0 S_{N_1(\gamma(0))} \gamma'(0) ,N_2, \tilde N) \\
& \ges |(\alpha+\beta) N_1(\gamma(0)) + \beta t_0 S_{N_1(\gamma(0))} \gamma'(0)| \cdot |\tilde N| \approx |N| \cdot |\tilde N|,
\end{split}
\]
where we have used the following consequence of \eqref{curva}:
\[
vol(N_1(\gamma(0)), S_{N_1(\gamma(0))} v ,N_2, \tilde N) \ges |\tilde N|,
\]
which holds true for any unit vector $v \in (T_{\gamma(0)} S_{1,\alpha_0})^\perp \subset T_{\gamma(0)} S_{1}$
and any vector $\tilde N \in d \calN_I$.
\end{proof}

Using a similar argument as above, one can easily establish the following dispersive estimate
\begin{equation} \label{dN}
| N_i(\zeta_1) - N_i(\zeta_2) | \approx d(S_{i,\alpha_1}, S_{i,\alpha_2})
\end{equation}
where $S_{i,\alpha_1}, S_{i,\alpha_2}$ are the leafs to which $\zeta_1, \zeta_2$ belong to, respectively. Here the distance between 
$S_{i,\alpha_1}$ and $S_{i,\alpha_2}$ can be defined either by using geodesics inside the hypersurface $S_i$ (using the induced metric from the ambient space $R^{n+1}$) or, equivalently, by using the classical distance between sets in $R^{n+1}$.

\section{Free waves, wave packets and tables on cubes} \label{CT}
 
In this section we collect some of the preparatory ingredients that are needed in the proof of our main result. The setup described here originates in the work of Tao on the bilinear restriction estimate \cite{Tao-BW}. 
All of the results here have been discussed in our previous works, see \cite{Be2} and \cite{Be3}. We do not repeat some of the proofs as they are similar to those found in \cite{Be2} and \cite{Be3}.
 
 \subsection{Rephrasing the problem in terms of free waves} \label{FW}
 
 We reformulate our problem in terms of free waves, this being motivated by the use of wave packets in the proof of Theorem \ref{mainT}. Once the wave packet decomposition is made and its properties are clear, the formalization of the problem as an evolution equation can be forgotten. 

Assume we are given a surface $S$ with a graph type parametrization $\zeta_{n+1}=\varphi(\xi)$ where $\xi = (\zeta_1,..,\zeta_n)$.
We rename the variable $\zeta_{n+1}$ by $\tau$, thus the equation of $S$ becomes $\tau=\varphi(\xi)$. We parametrize the physical space by $(x,t) \in \R^{n} \times \R$. We make the choice that $\tau$ is the Fourier variable corresponding to $t$, while
$\xi$ is the Fourier variable corresponding to $x$.  In what follows we use the convention that $\hat f$ denotes the Fourier transform of $f$ with respect to the $x$ variable.

We define the free wave $\phi = \calE f$ as follows
 \[
 \phi(x,t) = \calE f(x,t) = \int_{\R^n} e^{i(x\cdot \xi + t \varphi(\xi))} f(\xi) d\xi.
 \]
 Note that $\phi(0)=\check{f}$ and $\hat \phi(\xi,t) = e^{it \varphi(\xi)} \hat \phi(\xi,0)$. 
 We define the mass of a free wave by $M(\phi(t)):= \| \phi(t) \|^2_{L^2}$ and note that it is time independent:
 \[
 M(\phi(t)):= \| \phi(t) \|^2_{L^2}= \| \hat \phi(t) \|^2_{L^2} = \| \hat \phi(0) \|^2_{L^2}= \| \phi(0) \|^2_{L^2}= M(\phi(0)). 
 \]
The proof of \eqref{mainE} relies on estimating $\Pi_{i=1}^k \calE_i f_i $ on cubes on the physical side and see how this 
behaves as the size of the cube goes to infinity by using an inductive type argument with respect to the size of the cube.
Before we formalize this strategy, we note that at every stage of the inductive argument we re-localize functions both on the physical and frequency space, and, as a consequence, we need to quantify the new support on the frequency side. This will be done by using the using the margin of a function. 

We assume we are given a reference set $V$ inside which we want to keep all functions supported. If $f$ is supported in $U \subset V$ we define
the margin of $f$ relative to $V$ by
\[
\mbox{margin}(f) := \mbox{dist}(\mbox{supp} (f), V^c). 
\]
In terms of free waves $\phi = \calE f$, the margin is defined by
\[
\mbox{margin}(\phi(t)) := \mbox{dist}(\mbox{supp}_\xi (\hat \phi(t)), V^c)= \mbox{dist}(\mbox{supp} (f), V^c),
\]
where we have used that the Fourier support of $\hat \phi(t)$ is time independent and that $\hat \phi(0)=f$. In other words, 
the margin of a free wave is time independent. 

In practice, we work with $k$ different types of free waves, $\phi_i = \calE_i f_i, i=1,..,k$. They are assumed to be graphs with different phase functions
$\varphi_i$ and with potentially different ambient domain, that is $U_i$ are subsets of different subspaces isomorphic to $\R^n$ (for instance $U_i$ can be 
subsets of the hyperplanes $\xi_i=0$). The above construction changes only in making the choice of $\tau$ being the coordinate in the direction normal to the ambient hyperplane to which $U_i$ belongs to, while $\xi$ are the coordinates in the ambient hyperplane. Obviously, the margin of each $\phi_i$ is then defined with respect to some $V_i$ in the same ambient hyperplane. When choosing the reference sets $V_i$ we need to impose that the conditions i)-iii) hold true on $\Sigma_i(V_i)$ as well. 

Next, we prepare the elements that are needed for the induction on scale argument. Given that the estimate is trivial for $p=\infty$, it suffices to focus on the result above in the cases $p(k) < p \leq \frac2{k-1}$ and this is what we will do. Note that the exponent $\frac2{k-1}$ is precisely the one for which the $k$-linear restriction theory is expected to hold true without any curvature assumptions.

\begin{defin} Let $p(k) \leq p \leq \frac{2}{k-1}$. Given $R \geq C_0$ we define $A_p(R)$ to be the best constant for which the estimate
\begin{equation}
\| \Pi_{i=1}^k \phi_i   \|_{L^p(Q_R)} \leq A_p(R)  \Pi_{i=1}^k M(\phi_i)^\frac12. 
\end{equation}
holds true for all cubes $Q_R$ of size-length $R$, $\phi_i=\calE_i f_i$ and 
obeying the margin requirement
\begin{equation} \label{mrpg}
margin^i(\phi_i) \geq M-R^{-\frac14}, i=1,..,k.
\end{equation}
\end{defin}

The goal is to obtain an uniform estimate on $A_p(R)$ with respect to $R$. In the absence of the margin requirement above, 
$A_p(R)$ would be an increasing function. However, since the argument needs to tolerate 
the margin relaxation, we also define
\[
\bar A_p(R):= \sup_{1 \leq r \leq R} A_p(r) 
\]
and the new $\bar A_p(R)$ is obviously increasing with respect to $R$.  

Then \eqref{mainE}, and, as a consequence, the main result of this paper, Theorem \ref{mainT}, follow from the next result. 
\begin{prop} \label{keyP} Assume $0 < \epsilon < 1$. If $R \gg 2^{2C_0}$ and $R^{-\frac14+} \ll c \ll 1$, there exists $C(\epsilon)$
such that the following holds true:
\begin{equation} \label{ApR}
A_p(R) \leq (1+cC) \left( (1+cC)^p \left( \bar A_p(\frac{R}2) \right)^p   + 
\left( C(\epsilon) c^{-C} R^{ \frac{n+k+1}2(\frac1p-\frac{k}2 \cdot \frac{n+k-1}{n+k+1})+\epsilon} \right)^p \right)^\frac1p.
\end{equation}
\end{prop}
Deriving \eqref{mainE} from \eqref{ApR} is standard, see the corresponding argument in the trilinear case in \cite{Be3}. 
Thus we reduce the proof of Theorem \ref{mainT} to proving Proposition \ref{ApR}.  
 
\subsection{Tables on cubes}
 
Let $Q \subset \R^{n+1}$ be a cube of radius $R$. Given $j \in \N$ we split $Q$ into $2^{(n+1)j}$ cubes of size $2^{-j} R$ and denote this family by $\calQ_j(Q)$; thus we have $Q=\cup_{q \in \calQ_j(Q)} q$.  
If $j \in \N$ and $0 \leq c \ll 1$ we define the $(c,j)$ interior $I^{c,j}(Q)$ of $Q$ by
\begin{equation} \label{icj}
I^{c,j}(Q) := \bigcup_{q \in \calQ_j(Q)} (1-c) q.
\end{equation}
Given $j \in \N$ we define a table $\Phi$ on $Q$ to be a vector $\Phi=(\Phi^{(q)})_{q \in \calQ_j(Q)}$ and define its mass by
\[
M(\Phi) = \sum_{q \in \calQ_j(Q)} M(\Phi^{(q)}). 
\]
We define the margin of a table as the minimum margin of its components:
\[
margin(\Phi) = \min_{q \in \calQ_j(Q)} margin(\Phi^{(q)}).
\]
We recall from \cite{Be3} the following result:
\begin{lema} \label{AVL}
Assume $0 < p < \infty$, $R \gg 1$, $0 < c \ll 1$ and $f$ smooth. Given a cube $Q_R \subset \R^{n+1}$ of size $R$,
there exists a cube $Q$ of size $2R$ contained in $4 Q_R$ such that
\begin{equation} \label{avrg}
\| f \|_{L^p(Q_R)} \leq (1+cC) \|  f\|_{L^p(I^{c,j}(Q))}.
\end{equation}
\end{lema}

\subsection{Wave packets} \label{SWP}

In this section we formalize the wave packet construction for $k-1$-conical surfaces. We assume that $S$ is of $k-1$-conic type and has the graph-type parametrization $\Sigma: U \rightarrow S$, where $\Sigma(\xi)=(\xi,\varphi(\xi))$ and with foliations
$U =  \cup_\alpha U_\alpha$, $S =  \cup_\alpha S_\alpha$, $\Sigma(U_\alpha)=S_\alpha$. 

For the foliation $U=  \cup_\alpha U_\alpha$, we choose a system of coordinates $\bf x \rm: U \rightarrow \tilde U$
such that for each leaf $U_{\alpha}$, the coordinates of $U_{\alpha}$ are $\xi_k=constant,..,\xi_{n}=constant$. Let
$\tilde U'=\pi(\tilde U)$, where $\pi : \R^{n} \rightarrow \R^{n-k+1}$ is the projection $\pi(\xi_1,..,\xi_n)= (\xi_k,..,\xi_n)$. 
Let $\tilde{\mathcal{L}}$ be a maximal $r^{-1}$-separated subset of $\tilde U' \subset \R^{n-k+1}$. 
For each $\tilde \xi \in \tilde{\mathcal{L}}$, $\bx^{-1}(\cdot, \tilde \xi)$ is a leaf, that is
 $\bx^{-1}(\cdot, \tilde \xi)=U_\alpha$ for some $\alpha$.  In each such leaf we pick $\xi_T$ and call $\mathcal{L}$ to be 
the set obtained this way. It is not important which $\xi_T \in \bx^{-1}(\cdot, \tilde \xi)$ is chosen, since from condition ii) it follows that, for $\xi \in U_\alpha$, the normal $N(\Sigma(\xi))$ to $S$ is constant as $\xi$ varies inside the leaf $U_\alpha$. 
We denote by $U(\xi_T)$ the leaf $U_\alpha$ to which $\xi_T$ belongs and by $S(\xi_T)=\Sigma(U(\xi_T))$, the corresponding leaf on $S$. We note that $d(U(\xi_{T_1}), U(\xi_{T_2})) \approx d(\tilde \xi_1, \tilde \xi_2)$ which combined with \eqref{dN} gives
\begin{equation} \label{disp}
|N(\Sigma(\xi_{T_1})) - N(\Sigma(\xi_{T_2}))| \approx d(U(\xi_{T_1}), U(\xi_{T_2})) \approx d(\tilde \xi_1, \tilde \xi_2). 
\end{equation}

Let $L$ be the lattice $L=c^{-2} r \Z^n$. With $x_T \in L, \xi_T \in \mathcal{L}$ we define the tube $T=T(x_T,\xi_T):=\{ (x,t) \in \R^n \times \R: |x-x_T +  t \nabla \varphi(\xi_T) | \leq c^{-2} r \}$ and denote by $\calT$ the set of such tubes. One notices that $T$ is the $c^{-2}r$ neighborhood of the 
line passing through $(x_T,0)$ and direction $N(\Sigma(\xi_T))$. 

Associated to a tube $T \in \calT$, we define the cut-off $\tilde \chi_T$ on $\R^{n+1}$ by
\[
\tilde \chi_T(x,t)= \tilde \chi_{D(x_T -  t \nabla \varphi (\xi_T) ,t; c^{-2} r)}(x).
\] 
 We are ready to state the main result of this Section. 
\begin{lema} \label{LeWP} Let $Q$ be a cube of radius $R \gg 1$, let $c$ be such that $R^{-\frac14+} \ll c \les 1$ and let $J \in \N$ be such that $r =2^{-J} R \approx R^\frac12$. Let $\phi=\calE f$ be a free wave with $margin(\phi) > 0$. For each $T \in \calT$ there is a free wave
$\phi _T$, that is localized in a neighborhood of size $CR^{-\frac12}$ of the leaf $S(\xi_T)$ and obeying $margin(\phi_T) \geq margin(f)-CR^{-\frac12}$. The map $ f \rightarrow \phi_T$ is linear and 
\begin{equation} \label{lind}
\phi= \sum_{T \in \calT} \phi_{T}.
\end{equation}
If $\mbox{dist}(T,Q) \geq 4 R$ then 
\begin{equation} \label{ld}
\| \phi_T \|_{L^\infty(Q)} \ls c^{-C} dist(T,Q)^{-N} M(\phi)^\frac12.
\end{equation}
The following estimates hold true
\begin{equation} \label{qest}
\sum_{T} \sup_{q \in Q_J(Q)} \tilde \chi_T(x_q,t_q)^{-N} \| \phi_T \|^2_{L^2(q)} \ls c^{-C} r M(\phi)  
\end{equation}
and
\begin{equation} \label{q00e}
\left( \sum_{q_0} M( \sum_T m_{q_0,T} \phi_T)  \right)^\frac12
\leq (1+cC) M(\phi),
\end{equation}
provided that the coefficients $m_{q_0,T} \geq 0$ satisfy
\begin{equation} \label{q0e}
\sum_{q_0} m_{q_0,T}=1, \qquad \forall T \in \calT.
\end{equation}
\end{lema}
Originally, this type of wave packet decomposition was introduced by Tao in \cite{Tao-BW} in the context of bilinear restriction estimate for conical hypersufaces ($1$-conical in our language). The strength of this result lies in the use of the small parameter $c$ and the tight mass estimate \eqref{q00e}. In the case $c \approx 1$, the above decomposition is the standard wave packet decomposition. 

In the case of double-conical surfaces the analogue result was proved in \cite{Be3}. The argument for Lemma \ref{LeWP} 
is entirely similar to the results just mentioned and we will not duplicate it here. 

In the case $c \approx 1$, we will use the following variation of \eqref{qest}. Fix $N \in \N$; then for each tube $T \in \calT$, 
there are coefficients $c_{N,T}$ such that
\begin{equation} \label{qest1}
\sup_{q \in Q_J(Q)} \tilde \chi_T(x_q,t_q)^{-\frac{N}2} \| \phi_T \|_{L^2(q)} \ls r^\frac12 \cdot  c_{N}(T).   
\end{equation}
with the property that
\begin{equation} \label{massc}
\sum_{T \in \calT} c_{N}(T)^2 \les M(\phi). 
\end{equation}

\section{Table construction and the induction argument} \label{SI}

This section contains the main argument for the proof of Theorem \ref{mainT}. In Proposition \ref{NLP}
we construct tables on cubes: this is a way of re-organizing the information on one term, say $\phi_1$, at smaller scales
based on information from one of the other interacting terms, $\phi_i, i =2,..,k$. This type of argument is inspired by the work on the conic surfaces of Tao in \cite{Tao-BW}. Based on this table construction, we will prove the inductive bound claimed in Proposition \ref{keyP}.

\begin{prop} \label{NLP}

Let $Q$ be a cube of size $R \gg 2^{2C_0}$.
Assume $\phi_i=\calE_i f_i, i = 1,..,k$ have positive margin.  
Then there is a table $\Phi_1=\Phi_c(\phi_1, \phi_2,Q)$ with depth $C_0$
such that the following properties hold true: 
\begin{equation} \label{dec}
\phi_1= \sum_{q \in  \calQ_{C_0}(Q)} \Phi_1^{(q)}, 
\end{equation}
\begin{equation} \label{MPhi2}
margin(\Phi) \geq margin(\phi)-C R^{-\frac12}.
\end{equation}

\begin{equation} \label{PhiM2}
M(\Phi) \leq (1+cC) M(\phi),
\end{equation}
and for any $q',q'' \in \calQ_{C_0}(Q), q'\ne q''$
\begin{equation} \label{Phia2}
\| \Phi_1^{(q')} \Pi_{i=2}^k \|_{L^\frac{2}{k-1}((1-c)q'')} \ls c^{-C} R^{-\frac{n-k+1}4}  \Pi_{i=1}^k M^\frac12(\phi_i).
\end{equation}

\end{prop}

\begin{rem} \label{RNLP}
The above result is stated for scalar $\phi_1,..,\phi_k$, but it holds for vector versions as well. Most important 
is that we can construct $\Phi_1=\Phi_c(\phi_1, \Phi_2, Q)$ where $\Phi_2$ is a vector free wave and all its scalar components 
satisfy similar properties to the $\phi_2$ above. 
\end{rem}

\begin{rem} \label{RNLP1}
We note that $\Phi_1=\Phi_c(\phi_1, \phi_2,Q)$ means that the table  $\Phi_1$ is constructed from $\phi_1$, which is natural
in light of \eqref{dec}, and $\phi_2$. But it does not depend on $\phi_3,..,\phi_k$. Obviously, we could have constructed it from $\phi_1$
and $\phi_3$ (or any other $\phi_k$), ending with a different object. 
\end{rem}

In the proof below we use the results in Sections \ref{LocME} and \ref{CMS} in a crucial way. The reason we provide those results in later Sections is that,
at first reading, it is instructive to get the main points and the motivation for why the results in Sections \ref{LocME} and \ref{CMS} are necessary before getting 
too technical.

\begin{proof}

There are several scales involved in this argument. The large scale is the size $R$ of the cube
$Q$. The coarse scale is $2^{-C_0} R \gg R^\frac12$, this being the size of the smaller cubes
in $\calQ_{C_0}(Q)$ and the subject of the claims in the Proposition. Then there is the fine scale $r=2^{-j}R$ chosen such that
 $r \approx R^\frac12$. Notice that $r$ is the proper scale for wave packets corresponding to time scales $R$ and also
 that their scale is $c^{-2} r \ll 2^{-C_0} R$, the last one being the scale of cubes in $\calQ_{C_0}(Q)$. 
 
We use Lemma \ref{LeWP} with $J=j$ to construct the wave packet decomposition for $\phi_1$:
\[
\phi_1 = \sum_{T_1 \in \calT_1} \phi_{1,T_1}. 
\]
For any $q_0 \in \calQ_{C_0}(Q)$ and $T_1 \in \calT_1$ we define
\[
m_{q_0,T_1}:= \|  \tilde{\chi}_{T_1} \phi_2 \|^2_{L^2(q_0)}
\]
and 
\[
m_{T_1}:= \sum_{q_0 \in \calQ_{C_0}(Q)} m_{q_0, T_1}.
\]
Based on this we define
\begin{equation} \label{dPq0}
\Phi^{(q_0)}_1:= \sum_{T_1} \frac{m_{q_0,T_1}}{m_{T_1}} \phi_{1,T_1}. 
\end{equation}

By combing the definitions above with the decomposition property \eqref{lind}, we obtain 
\[
\phi_1= \sum_{q_0 \in \calQ_{C_0}(Q)} \Phi_1^{(q_0)}
\]
thus justifying \eqref{dec}. 

The margin estimate \eqref{MPhi2} follows from the margin estimate on tubes provided by Lemma \ref{LeWP}. 
The coefficients $m_{q_0,T_1}$ satisfy \eqref{q0e}, thus the estimate \eqref{PhiM2} follows from 
 \eqref{q00e}. 

All that is left to prove is \eqref{Phia2}, which is equivalent to 
\begin{equation} \label{red1}
\left( \sum_{q \in \calQ_j(Q):d(q,q_0) \ges cR} \| \Phi_1^{(q_0)} \Pi_{i=2}^k \phi_i \|^{\frac2{k-1}}_{L^{\frac2{k-1}}(q)}\right)^{\frac{k-1}2} 
\ls c^{-C} r^{-\frac{n-k+1}2} \Pi_{i=1}^k M(\phi_i).
\end{equation}

Note that the cubes $q$ are selected at the finer scale dictated the size of cubes in $\calQ_j(Q)$. In the definition of 
$\Phi_1^{(q_0)}$, see \eqref{dPq0}, we have the full family $\calT_1$. In the above estimate, we estimate the output inside $q$, 
thus, in light of  \eqref{ld}, the terms $\phi_{T_1}$ with $T_1 \cap q \ne \emptyset$ are the ones that really matter. Indeed, if we split
\[
\Phi^{(q_0)}_1:= \sum_{ T_1 \cap q \ne \emptyset} \frac{m_{q_0,T_1}}{m_{T_1}} \phi_{1,T_1}  
+ \sum_{k \in \N} \sum_{d(T_1,q) \approx 2^k c^{-2} r} \frac{m_{q_0,T_1}}{m_{T_1}} \phi_{1,T_1}, 
\]
we can use \eqref{triangle} to reduce the problem to estimating each term above in the first sum. Indeed, in light of \eqref{qest}, the contributions of terms from the second sum come with additional decay $2^{-kN}$, which, for $N$ large enough, can be easily estimated. Thus it suffices to prove the estimate \eqref{red1} with $\Phi^{(q_0)}_1$ replaced by the first sum above.

For fixed $q$, it is a straightforward exercise to check that Setup \ref{C} in Section \ref{LocME} is satisfied: simply let $J= \{T_1 \in \calT_1: T_1 \cap q \neq \emptyset\}$ 
and let $\phi_{T_1}= \calE_{1,T_1} f_{1,T_1}$. Thus
we can  invoke \eqref{MECloc}  to obtain
\[
\begin{split}
 \| \left( \sum_{T_1 \cap q \ne \emptyset} \frac{m_{q_0,T_1}}{m_{T_1}} \phi_{T_1}  \right)  \Pi_{i=2}^k \phi_i  \|_{L^{\frac2{k-1}}(q)}
 \ls C(\epsilon) r^{-\frac{n+1}2} r^\epsilon \sum_{T_1 \cap q \ne \emptyset}  \frac{m_{q_0,T_1}}{m_{T_1}} \| \tilde \chi_q \phi_{1,T_1} \|_{L^2}  \Pi_{i=2}^{k} \| \tilde \chi_q \phi_i \|_{L^2}.
\end{split}
\]
Since $m_{q_0,T_1} \leq m_{T_1}$, then $\frac{m_{q_0,T_1}}{m_{T_1}} \leq \frac{m^\frac12_{q_0,T_1}}{m^\frac12_{T_1}}$; from this we obtain:
\[
\begin{split}
&  \sum_{ T_1 \cap q \ne \emptyset}  \frac{m_{q_0,T_1}}{m_{T_1}} \| \phi_{1,T_1} \tilde \chi_q \|_{L^2}\\
 \ls & 
\left( \sum_{T_1 \cap q \ne \emptyset} \frac{\|\phi_{1,T_1} \tilde\chi_q \|^2_{L^2}}{m_{T_1} \tilde \chi_{T_1}(x_q,t_q)} \right)^\frac12 
 \left(  \sum_{T_1 \cap q \ne \emptyset}  m_{q_0,T} \tilde \chi_{T_1}(x_q,t_q)  \right)^\frac12.
\end{split}
\]
Next we claim the following estimate
\begin{equation} \label{keyw}
\begin{split}
\sum_{T_1 \in \calT_1} m_{q_0,T_1} \tilde \chi_{T_1}(x_q,t_q) & \ls  \| \tilde \chi_{S(q)} \phi_2 \|_{L^2}^2. 
\end{split}
\end{equation}
Using the definition of $m_{q_0,T_1}$ we identify  the function 
\[
\tilde \chi_{S(q)}= ( \sum_{T_1 \in \calT_1} \tilde \chi(x_q,t_q) \tilde \chi_{T_1} ) \chi_{q_0}
\]
which makes  \eqref{keyw} hold true. Here the surface $S(q)$ is the translate by $c(q)$ of the neighborhood of size $r$ of cone of normals at $S_1$, which we denote by $\mathcal{CN}_1:=\{\alpha N_1(\zeta), \zeta \in S_1, \alpha \in \R \}$.  It is important to note that we do not consider the whole cone but only the part with $cR \leq \alpha \ls R$. 
$\tilde \chi_{S(q)}$ has the following decay property:
\[
\tilde \chi_{S(q)} (x,t) \ls c^{-4} \left( 1+ \frac{d((x,t), S(q))}{c^{-2}r}\right)^{-N}.
\]
This is a consequence of the fact that the tubes $T_1$ passing thorough $q$ separate inside $q_0$ and of the separation between $q$ and $q_0$, which is quantified by $d(q,q_0) \ges cR$. Quantitatively speaking, given a point in $q_0$ close to $S(q)$, there are $\ls c^{-4}$ tubes $T_1$ passing through the point and $q$ - this follows from the dispersion estimate \eqref{disp} and 
the geometry of the family of tubes $\calT_1$.

We define
\[
\begin{split}
& A(q)= \left( \sum_{T_1 \cap q \ne \emptyset} \frac{\|\phi_{1,T_1} \tilde\chi_q \|^2_{L^2}}{m_{T_1} \tilde \chi_{T_1}(x_q,t_q)} \right)^\frac12, \quad B(q)=   \| \tilde \chi_{S(q)} \phi_2 \|_{L^2} \\
& C(q)= \| \tilde \chi_q \phi_2 \|, \quad D(q)= \Pi_{i=2}^{k} \| \tilde \chi_q \phi_i \|_{L^2}.
\end{split}
\]
To conclude with the proof of \eqref{red1}, it suffices to show 
\[
\left( \sum_{q \in \calQ_j(Q):d(q,q_0) \ges cR} A(q)^\frac{2}{k-1} B(q)^\frac2{k-1} C(q)^\frac2{k-1} D(q)^\frac2{k-1} \right)^\frac{k-1}2 
\les r^{\frac{k}2} r^\epsilon \Pi_{i=1}^k M(\phi_i)^\frac12.
\]
This will be a consequence of the following two inequalities
\begin{equation} \label{mes2}
\left( \sum_{q \in \calQ_j(Q):d(q,q_0) \ges cR} A(q)^2 B(q)^2  \right)^\frac12 
\les r^{\frac{1}2}  M(\phi_1)^\frac12,
\end{equation}
and
\begin{equation} \label{mes3}
\left( \sum_{q \in \calQ_j(Q):d(q,q_0) \ges cR} C(q)^\frac2{k-2} D(q)^\frac2{k-2} \right)^\frac{k-2}2 
\les r^{\frac{k-1}2} r^\epsilon \Pi_{i=2}^k M(\phi_i)^\frac12.
\end{equation}
The proof of \eqref{mes2} is similar to the one we used in the bilinear and trilinear theory, see \cite{Be2, Be3}. 
By rearranging the sum, it suffices to show
\[
\sum_{T_1} \sum_{q \cap T_1 \ne \emptyset}   \frac{\|\phi_{1,T_1} \tilde \chi_q \|^2_{L^2} 
\| \phi_2 \tilde \chi_q \|^2_{L^2}}{m_{T_1} \tilde \chi_{T_1}(x_q,t_q)} \ls r M(\phi_1).
\]
The inner sum is estimated as follows:
\[
\sum_{q \cap T_1 \ne \emptyset}   \frac{ \| \phi_2 \tilde \chi_q \|^2_{L^2}}{m_{T_1} \tilde \chi_{T_1}(x_q,t_q)}
\ls   \frac{ \| \phi_2 \tilde \chi_{T_1} \|^2_{L^2}}{m_{T_1}} \ls 1,
\]
and the outer one is estimated by
\[
\sum_{T_1} \sup_{q} \|\phi_{1,T_1} \tilde \chi_q \|^2_{L^2}  \ls r \sum_{T_1} M(\phi_{1,T_1}) \ls r M(\phi_1),
\]
which is obvious given the size of $q$ in the $x_1$-direction is $\approx r$ and the mass of $\phi_{1,T_1}$
is constant across slices in space with $x_1=constant$. 

In proving \eqref{mes3}, we can take advantage of the fast decay of $\tilde \chi_q$ away from $q$ and of $\tilde \chi_{S(q)}$
away from $S(q)$, and at the cost of picking factors of type $c^{-C}$, it suffices to show
\begin{equation} \label{CK}
\| \| \phi_2 \|_{L^2(S(q))}  \Pi_{i=3}^k\| \phi_i \|_{L^2(q)} \|_{l^{\frac2{k-2}}_q} 
\les r^{\frac{k-1}2} r^\epsilon \Pi_{i=2}^k M(\phi_i).
\end{equation}
The $l^{\frac2{k-2}}_q$ norm is computed over the set of $q \in \calQ_j(Q)$, the set of cubes of size $r$ contained in the larger cube of size $r^2$. 
This estimate is the subject of Theorem \ref{NewRE} in Section \ref{CMS}. The statement of Theorem \ref{NewRE} requires $S$ to have certain properties in relation 
to the other surfaces $S_2,..S_k$, see \bf P1, P2 \rm at the beginning of Section \ref{CMS}. The fact that $S$ satisfies these properties follows from 
Lemma \ref{GL}. 

\end{proof}

\begin{proof}[Proof of Proposition \ref{keyP}] This is entirely similar to the argument used in \cite{Be2} and \cite{Be3}, see the corresponding proofs there.
\end{proof}

We have finished the proof of our main result Theorem \ref{mainT}. Obviously we owe a justification for some estimates used in the body of the proof of Proposition \ref{NLP} and
this what will be covered in the next two Sections of the paper. 

\section{Second part: the multilinear estimate revised} \label{SePa}
We have arrived at the middle point in this article. In the first half, Sections \ref{intr} through \ref{SI}  we have proved the main result, Theorem \ref{mainT}. In the second part, and Sections \ref{LocME} and \ref{CMS}, we provide some of the supporting details used in the proof of Theorem \ref{mainT}. However we think that these are not just technical results, but they may be of independent interest. 

We point out a major difference between the hypothesis used in the two parts. For Theorem \ref{mainT} we assume the particular 
foliation structure and curvature condition described by conditions i)-iii). In the second part, Sections \ref{LocME} and \ref{CMS}, 
we provide results in a general setup which we describe bellow. 

We are given $k$ smooth hypersurfaces $S_i = \Sigma_i(U_i)$ with smooth parameterizations 
$\Sigma_i$. These should be seen as new surfaces, different than the ones for which Theorem \ref{mainT} states a result; 
most important difference is that 
$S_i, i=1,..,k$ used here are generic, in other words they are not assumed to have a foliation structure, nor curvature properties as the surfaces in our main result, Theorem \ref{mainT}.

We assume the transversality condition: there exists $\nu >0$ such that
\begin{equation} \label{normal2}
vol (N_1(\zeta_1), .., N_{k}(\zeta_{k})) \geq \nu,
\end{equation}
for all choices $\zeta_i \in \Sigma_i(U_i)$. Here by $vol (N_1(\zeta_1), .., N_{k}(\zeta_{k}))$ we mean the volume of the $k$-dimensional parallelepiped spanned by the vectors $N_1(\zeta_1), .., N_{k}(\zeta_{k})$. 

Each of these (parametrization of) hypersurfaces generates the corresponding $\calE_i$ operator:
\[
\calE_i f(x) = \int_{U_i} e^{i x \cdot \Sigma_i(\xi)} f(\xi) d\xi.
\]

\section{The multilinear estimate: localization and superposition} \label{LocME}
In this section we provide the proof of a localized version of the multilinear estimate. The motivation comes from the argument in the previous section. The proofs build on the ideas introduced in \cite{Be1}
and later refined in \cite{Be2}.

We work under the setup described in Section \ref{SePa}. Given $N_{k+1},..,N_{n+1}$ unit vectors, we introduce the following transversality condition: there exists $\nu > 0$ such that
\begin{equation} \label{normal}
| det(N_1(\zeta_1), .., N_{k}(\zeta_{k}), N_{k+1},.., N_{n+1})| \geq \nu
\end{equation}
for all choices $\zeta_i \in \Sigma_i(U_i)$. 

Assume that $\Sigma_1 (supp f_1) \subset  B(\calH_1,\mu)$, where $B(\calH_1,\mu)$ is the neighborhood of size $\mu$ of the $k$-dimensional affine subspace $\calH_1$. Assume that $|N_1(\zeta_1)- \pi_{\calH_1} N_1(\zeta_1)| \les \mu, \forall \zeta_1 \in \Sigma_1 (supp f_1)$,
where $\pi_{\calH_1} : \R^{n+1} \rightarrow \calH_1$ is the projection onto $\calH_1$. In addition assume that if $N_{i}, i=k,..,n+1$ is a basis of the normal space $\calH^\perp_1$ to $\calH_1$, then $N_1(\zeta_1),.., N_k(\zeta_k), N_{k+1},..,N_{n+1}$ are transversal in the sense \eqref{normal}. Under these hypothesis we proved in  \cite[Theorem 1.3]{Be2} that
\begin{equation} \label{AMRE}
\| \Pi_{i=1}^k \calE_i  f_i  \|_{L^\frac2{k-1}(B(0,r))} \leq C(\epsilon) \mu^{\frac{n-k+1}2} r^\epsilon 
\Pi_{i=1}^k \| f_{i} \|_{L^2(U_{i})} .
\end{equation}

The multilinear estimate \eqref{AMRE} is a statement about the product of some functions in 
$L^\frac{2}{k-1}$. It is very natural to ask how does this estimate behaves with respect to superpositions of one factor, that is replacing $f_1$ by $\sum_{\alpha} f_{1,\alpha}$.  If $\frac{2}{k-1} \geq 1$, then the triangle inequality holds true in $L^\frac{2}{k-1}$ and the answer is simple: in a sublinear fashion. If $\frac{2}{k-1} <1$, the triangle inequality fails in  $L^\frac{2}{k-1}$ and the sublinearity cannot be argued in the same way. However 
\[
\begin{split}
\|  \calE_1 ( \sum_\alpha f_{1,\alpha})  \Pi_{i=2}^{k} \calE_i f_i \|_{L^\frac{2}{k-1}(B(0,R))} & \leq C(\epsilon) R^\epsilon \|  \sum_\alpha f_{1,\alpha} \|_{L^2} \Pi_{i=2}^{k} \| f_i \|_{L^2(U_i)} \\
& \leq C(\epsilon) R^\epsilon \sum_\alpha \|   f_{1,\alpha} \|_{L^2}  \Pi_{i=2}^{k} \| f_i \|_{L^2(U_i)}
\end{split}
\]
and this indicates again sublinear behavior with respect to superpositions of one input. In the above the set of indexes $\alpha$ is taken to be of finite cardinality 
(to avoid unnecessary distractions) and the key point is that the estimate is independent of the cardinality of this set.

The main question is whether the sublinearity aspect of the estimate holds true for the refinement \eqref{AMRE} of the multilinear estimate. An a posteriori argument as above will fail to give the optimal result when each term $f_{1,\alpha}$ has good localization properties, but $\sum_{\alpha} f_{1,\alpha}$ does not have such localization properties. 

\begin{set} \label{C}
We are given $J$, a finite set, and open, bounded and connected sets $U_{1,\alpha} \subset \calH_{1,\alpha}, \forall \alpha \in J$, where $\calH_{1,\alpha}$ are affine hyperplanes. 
For each $\alpha \in J$ we assume the following: there are $k$-dimensional hyperplanes $\calH'_{1,\alpha}$ with the property that $S_{1,\alpha} = \Sigma_{1,\alpha}(U_{1,\alpha}) \subset B( \calH'_{1,\alpha},\mu)$, the neighborhood of size $\mu$ of $\calH'_{1,\alpha}$. The following property holds $|N_1(\zeta_1)- \pi_{\calH'_{1,\alpha}} N_1(\zeta_1)| \les \mu, \forall \zeta_1 \in S_{1,\alpha}$, where $\pi_{\calH'_{1,\alpha}} : \R^{n+1} \rightarrow \calH'_{1,\alpha}$ is the projection onto $\calH'_{1,\alpha}$.
Let $\tilde \calH_{1,\alpha}= \calH_{1,\alpha} \cap \calH'_{1,\alpha}$ be the $k-1$-dimensional affine subspace $\tilde \calH_{1,\alpha} \subset \calH_{1,\alpha}$; we also assume that $U_{1,\alpha} \subset B( \tilde \calH_{1,\alpha},\mu)$.  

We assume that $S_{1,\alpha} \subset S_1=\Sigma_1(U_1), \forall \alpha \in J$, and  $S_1$ satisfies the global property:
there is an orthonormal set of vectors $N_{i}, i=k,..,n+1$ such that \eqref{normal} is satisfied. 

For each $\alpha \in J$, we assume that if $N_{i}, i=k+1,..,n+1$ is a basis of the normal space 
$\tilde \calH^\perp_{1,\alpha} \subset \calH_{1,\alpha}$, then $N_1(\zeta_1),.., N_k(\zeta_k), N_{k+1},..,N_{n+1}$ are transversal in the sense \eqref{normal}. 
 \end{set}

For each $\alpha \in J$ we define
\[
\calE_{1,\alpha} f(x) = \int_{U_{1,\alpha}} e^{i x \cdot \Sigma_{1,\alpha}(\xi)} f(\xi) d\xi. 
\]
Without restricting the generality of the problem, we can assume that $S_{1,\alpha}$ are of graph type, that is 
 $\Sigma_{1,\alpha}(\xi^\alpha)=(\xi^\alpha,\varphi_{1,\alpha}(\xi^\alpha))$,
where $\xi^\alpha$ is the coordinate in $\calH_{1,\alpha}$. In addition, for each $\alpha$, we pick and fix some $\eta_{1,\alpha} \in U_{1,\alpha}$. 

The next result states how the multilinear estimate behaves with respect to superposition of localized functions.  

\begin{theo} \label{MB2} We assume the Setup \ref{C}. Let $\mu, R > 0$ be such that $R \leq \mu^{-1}$. 
Then for any $\epsilon > 0$, there is $C(\epsilon)$ such that the following holds true
\begin{equation} \label{Lf2}
\begin{split}
& \| (\sum_{\alpha} \calE_{1,\alpha}  f_{1,\alpha}) \Pi_{i=2}^{k} \calE_i f_i \|_{L^\frac{2}{k-1}(B(0,R))} \\
\leq & C(\epsilon) \mu^{\frac{n+1-k}2} R^\epsilon \left( \sum_{\alpha} \| f_{1,\alpha} \|_{L^2(U_{1,\alpha})} \right) \Pi_{i=2}^{k} \| f_i \|_{L^2(U_i)}.
\end{split}
\end{equation}
\end{theo}
In Section \ref{SI} we have used the following consequence of the above Theorem: 
\begin{corol} \label{MBcor} We assume the Setup \ref{C}. Assume that $\mu \approx r^{-1}$ and $q$ is a cube of size $\approx r$.
Then for any $\epsilon > 0$, there is $C(\epsilon)$ such that the following holds true
\begin{equation} \label{MECloc}
\begin{split}
& \| (\sum_{\alpha} \calE_{1,\alpha}  f_{1,\alpha}) \Pi_{i=2}^{k} \calE_i f_i \|_{L^\frac{2}{k-1}(q)} \\
\leq & C(\epsilon) r^{-\frac{n+1}2} r^\epsilon \left( \sum_{\alpha} \| \tilde \chi_q \calE_{1,\alpha} f_{1,\alpha} \|_{L^2} \right) 
\Pi_{i=2}^{k} \| \tilde \chi_q \calE_i f_i \|_{L^2}.
\end{split}
\end{equation}
\end{corol}
We note that the apparent gain of a factor of $r^{-\frac{k}2}$ in this Corollary over the result in Theorem \ref{MB2} has to do with 
replacing $\| f_i \|_{L^2(U_i)}$ by $\| \tilde \chi_q \calE_i f_i \|_{L^2}$. 

The result of the Corollary is not an immediate consequence of the Theorem \ref{MB2}; but it will follow easily from the arguments 
used in the proof of Theorem \ref{MB2}.

The plan is the following: we introduce some notation specific to this section and then we proceed with the proof of the above two results. 

\subsection{Notation} \label{NOT} 
 
Assume $\calH_1 \subset \R^{n+1}$ is a hyperplane passing through the origin. Let $N_1$ be its normal and let $\pi_{N_1}: \R^{n+1} \rightarrow \calH_1$
the associated projection along the normal $N_1$. We denote by $\mathcal{F}_1: \calH_1 \rightarrow \calH_1$ the Fourier transform and by $\mathcal{F}_1^{-1}$ the inverse Fourier transform. We denote the variables in $\R^{n+1}$ by $x=(x_1,x')$, where $x_1$ is the coordinate along $N_1$ and $x'$ is the coordinate along $\calH_1$.
 We denote by $\xi'$ the Fourier variable corresponding to $x'$. For $f: U_1 \subset \calH_1 \rightarrow \C$, $f \in L^2(U_1)$ operator $\calE_1$ 
 takes the form 
\begin{equation} \label{E1}
\calE_1 f (x)= \int_{U_1} e^{i (x' \xi' + x_1 \varphi_1(\xi'))} f(\xi') d\xi'.
\end{equation}
 We define the differential operator $\nabla \varphi_1(\frac{D'}i)$ to be the operator with symbol $\nabla \varphi_1(\xi')$. 
 The following  commutator estimate holds true
\begin{equation} \label{com}
(x'-x'_0-x_1 \nabla \varphi_1(\frac{D'}i))^N \calE_1 f =  \calE_1 (\mathcal{F}_1( (x'-x'_0)^N  \mathcal{F}_1^{-1} f)), \quad \forall N \in \N. 
\end{equation}
This is a direct computation using \eqref{E1} and it suffices to check it for $N=1$. 
The role of \eqref{com} will be to quantify localization properties of $\mathcal{F}_1^{-1} f$ on hyperplanes with $x_1=constant$.

We are given $\calH_i, i=1,..,k$ , reference hyperplanes that are used in defining $\calE_i f_i, i=1,..,k$. 
Their normals are denoted by $N_i, i=1,..,k$, respectively.  Note that since $S_{1,\alpha} \subset S_1, \forall \alpha \in J$, it follows that $N_1$
is transversal to all $\calH_{1,\alpha}$. We then pick unit vectors $N_{k+1},..,N_{n+1}$ such that \eqref{normal} is satisfied. 

We construct $\mathcal{L}:=\{ z_1 N_1 + ... +z_{n+1} N_{n+1}: (z_1,..,z_{n+1}) \in \Z^{n+1} \}$ to be the oblique lattice in $\R^{n+1}$ generated by the unit vectors $N_1,..,N_{n+1}$. In each $\calH_i, i=2,..,k$ we construct the induced lattice $\mathcal{L}(\calH_i)=\pi_{N_i} (\mathcal{L})$; this is a lattice since the projection is taken along a direction of the original lattice $\mathcal{L}$. 

Given $r > 0$ we define $\calC(r)$ to be the set of of parallelepipeds of size $r$ in $\R^{n+1}$ relative to the lattice $\mathcal{L}$; a parallelepiped in $\calC(r)$ has the following form $q(\bj): = [r (j_1-\frac12), r(j_1+\frac12)] \times .. \times [r (j_{n+1}-\frac12), r(j_{n+1}+\frac12)]$
where $\bj=(j_1,..,j_{n+1}) \in \Z^{n+1}$. For such a parallelepiped we define $c(q)=r \bj=(r j_1,.., r j_{n+1}) \in r\mathcal{L}$ to be its center.
For each $i=2,..,k$, we let $\calC\calH_i(r) =\pi_{N_i} \calC(r)$ be the set of parallelepipeds of size $r$ in the hyperplane $\calH_i$.
 Given two parallelepipeds $q,q' \in \calC(r)$ or $\calC \calH_i(r)$ we define
$d(q,q')$ to be the distance between them when considered as subsets of the underlying space, let it be $\R^{n+1}$ or $\calH_i$.  

For each $i \in\{ 2,..,k \}$, $r> 0$ we define the linear operator $\mathcal{T}_i : \calH_i \rightarrow \calH_i$
to be the operator that takes $\mathcal{L}(\calH_i)$ to the standard lattice $\Z^{n}$ in $\calH_i$. Then for each $q \in \calC \calH_i(r)$, define
$\chi_q: \calH_i \rightarrow \R$ by
\[
\chi_{q}(x) = \eta_0 (\mathcal{T}_i(\frac{x-c(q)}r))
\]
Notice that $\mathcal{F}_i \chi_{q}$ has Fourier support in the ball of radius $\les r^{-1}$. By the Poisson summation formula and properties of $\eta_0$, 
\begin{equation} \label{pois}
\sum_{q \in \calC \calH_i(r)} \chi_{q}=1.
\end{equation}
Using the properties of $\chi_q$, a direct exercise shows that for each $N \in \N$, the following holds true
\begin{equation} \label{SN}
\sum_{q \in \calC \calH_i(r)}  \| \la \frac{x-c(q)}r \ra^{N} \chi_{q} g \|_{L^2}^2 \les_N \| g \|_{L^2}^2
\end{equation}
for any $g \in L^2(\calH_i)$. Here, the variable $x$ is the argument of $g$ and belongs to $\calH_i$.

Next we turn our attention to similar objects corresponding to the more complex family indexed by $\alpha \in J$. Given $\tilde  \calH_{1,\alpha} \subset \calH_{1,\alpha}$ a subspace of dimension $k-1$, we let $\tilde \pi_\alpha : \calH_{1,\alpha} \rightarrow \tilde  \calH_{1,\alpha}$ be orthogonal projector onto $\tilde  \calH_{1,\alpha}$.  We denote by $(\tilde \calH_{1,\alpha})^\perp$ the normal subspace to $\calH_{1,\alpha}$. 

 We let $\pi_{N_{1,\alpha}}: \R^{n+1} \rightarrow \calH_{1,\alpha}$ be the projector onto $\calH_{1,\alpha}$ 
 and $\tilde \pi_{1,\alpha} :=  \tilde \pi_\alpha \circ \pi_{N_{1,\alpha}} : \R^{n+1} \rightarrow \tilde \calH_{1,\alpha}$ be the projector onto $\tilde \calH_{1,\alpha}$. We define the latices $\mathcal{L}(\calH_{1,\alpha})=\Z^n$ inside $\calH_{1,\alpha}$ and 
 $\mathcal{L}(\tilde \calH_{1,\alpha})=\Z^{k-1}$
inside $\tilde \calH_{1,\alpha}$ with respect to orthonormal basis in each case; they are constructed such that $\tilde \pi_\alpha (\mathcal{L}(\calH_{1,\alpha}))=\mathcal{L}(\tilde \calH_{1,\alpha})$ - this holds true if the orthonormal basis in $\tilde \calH_{1,\alpha}$ is a subset of the orthonormal basis in $\calH_{1,\alpha}$.

Inside the subspace $\calH_{1,\alpha}$ we construct $\calC_{1,\alpha}(r)$ to be the set of cubes of size $r$ centered at points from the lattice $r \mathcal{L}(\calH_{1,\alpha})$ and sides parallel to the directions of the lattice. Inside the subspace $\tilde \calH_{1,\alpha}$ we construct $\tilde \calC_{1,\alpha}(r)$ be the set of cubes of size $r$ centered at points from the lattice $r \mathcal{L}(\tilde \calH_{1,\alpha})$ and sides parallel to the directions of the lattice. Therefore $\tilde \calC_{1,\alpha}(r)=\tilde \pi^\alpha \calC_{1,\alpha}(r)$. 
Then we define $ \mathfrak{S}_{1,\alpha}(r)$ to be the set of infinite cubical strips
$\mathfrak{s}=q \times (\tilde \calH_{1,\alpha})^\perp \subset \calH_{1,\alpha}$, where $q \in \tilde \calC_{1,\alpha}(r)$. We denote
by $c(\mathfrak{s}):=c(q) \subset r \mathcal{L}(\tilde \calH_{1,\alpha})$, the center of the strip.  We note that given $q_1, q_2 \in \calC_{1,\alpha}(r)$, they belong to the same cubical strip in 
$\mathfrak{S}_{1,\alpha}(r)$ if and only if $\tilde \pi_\alpha q_1=\tilde \pi_\alpha q_2$. For $q \in \calC_{1,\alpha}(r)$, we let $\mathfrak{s} (\tilde \pi_\alpha q)$ be the infinite cubical strip it belongs to as a subset in $\mathfrak{S}_{1,\alpha}(r)$. Given a strip $\mathfrak{s} \in \mathfrak{S}_{1,\alpha}(r)$ we define $\chi_{\mathfrak{s} }: \calH_{1,\alpha} \rightarrow \R$
\[
\chi_{\mathfrak{s}}(x) = \eta_0 (\frac{\tilde \pi_{1,\alpha}(x)-c(\mathfrak{s})}{r})
\]
where, by abusing notation, $\eta_0: \R^{k-1} \rightarrow \R$ is entirely similar to the $\eta_0$ introduced in Section \ref{NOT}, expect that it acts on $\R^{k-1}$ instead of $\R^{n}$. A key property of $\chi_{\mathfrak{s}}$ is that it is constant in directions from the subspace 
$(\tilde \calH_{1,\alpha})^\perp$.

One unpleasant feature of the above construction is that the lattice $\mathcal{L}$ does not project exactly into the latices $\mathcal{L}(\tilde \calH_{1,\alpha})$ via $\tilde \pi_{1,\alpha}$; similarly $\calC_\alpha(r)$ does not project well into $\tilde \calC_{1,\alpha}(r)$ via 
$\tilde \pi_{1,\alpha}$. This is an inherent feature of the fact that there are too many subspaces $\tilde \calH_{1,\alpha}$. As a consequence, given $q \in \calC(r)$, it is not necessarily true that $\tilde \pi_{1,\alpha}(q) \in \tilde \calC_{1,\alpha}(r)$; however $\tilde \pi_{1,\alpha}(q)$ intersects a finite number of 
 $q' \in \tilde \calC_{1,\alpha}(r)$. Abusing notation, we define 
 \[
 \mathfrak{s}^\alpha (\tilde \pi_{1,\alpha} (q))=\bigcup_{q' \in \tilde \calC_{1,\alpha}(r): q' \cap \tilde \pi_{1.\alpha} (q) \neq \emptyset} \mathfrak{s}^\alpha (q'),
 \]
 the strip generated by the projection of $q$ onto $\tilde \calH_{1,\alpha}$. 
 
 Recalling that $\mathcal{L}:=\{ z_1 N_1 + ... +z_{n+1} N_{n+1}: (z_1,..,z_{n+1}) \in \Z^{n+1} \}$, we denote the coordinates of a point in the lattice by $(z_1,..,z_{n+1})$
 and define
\[
\| g \|_{l^\infty_{z_1,z_{k+1},..,z_{n+1}} l^2_{z_2,..,z_k}(\mathcal{L})} = sup_{z_1,z_{k+1},..,z_{n+1}} \| g(z_1, \cdot, z_{k+1},..,z_{n+1}) \|_{l^2_{z_2,..,z_k}}
\]
where $\cdot$ stands for the variables $z_2,..,z_k$ with respect to which $l^2$ is computed. 

With this notation in place we have the following result:

\begin{lema} \label{Le2}
Assume $g_1 \in l^\infty_{z_1,z_{k+1},..,z_{n+1}} l^2_{z_2,..,z_k}(\mathcal{L})$ and $g_i \in l^2(\mathcal{L}(\calH_i)), i=2,..,k$. Then the following holds true
\begin{equation} \label{LWd2}
\| g_1(z) \Pi_{i=2}^{k} g_i(\pi_{N_i}(z)) \|_{l^{\frac2{k-1}}(\mathcal{L})} \ls \|g_1\|_{l^\infty_{z_1,z_{k+1},..,z_{n+1}} l^2_{z_2,..,z_k}(\mathcal{L})}  \Pi_{i=2}^{k} \| g_i \|_{l^2(\mathcal{L}(\calH_i))}. 
\end{equation}
\end{lema}

\begin{proof} The function $g_i \circ \pi_{N_i} $ is independent of the $z_i$ variable, therefore it holds true that
$g_i \circ \pi_{N_i} \in l^2_{z_1,z_{k+1},..,z_{n+1}} l^2_{z_2,..,z_{i-1}}   l^\infty_{z_i} l^2_{z_{i+1},..,z_{k}}$ and 
\[
\| g_i \circ \pi_{N_i} \|_{l^2_{z_1,z_{k+1},..,z_{n+1}} l^2_{z_2,..,z_{i-1}}   l^\infty_{z_i} l^2_{z_{i+1},..,z_{k}}} \leq \| g_i \|_{l^2(\mathcal{L}(\calH_i))},
\]
where then norms $l^2_{z_1,z_{k+1},..,z_{n+1}} l^2_{z_2,..,z_{i-1}}   l^\infty_{z_i} l^2_{z_{i+1},..,z_{k}}$ are defined in the standard fashion. 
Then the result is a direct consequence of the H\"older inequality in its discrete version. 

\end{proof}

\subsection{Proofs of the main results} 

\begin{proof}[Proof of Theorem \ref{MB2}] The argument is based on an induction on scales. Given a $0 < \delta \ll 1$, 
we break the surfaces into smaller pieces of diameter $\les \delta$. A result on the smaller scales is converted to a result at the original scale at the cost of a large power of $\delta^{-1}$, which is absorbed into $C(\epsilon)$. Thus, the focus will be on providing a result in the context of surfaces with diameter less than $\delta$.

We run an induction with respect to the size of the cube were estimates are made. We show that passing from an estimate on cubes of size $R$
to estimate on cubes of size $\delta^{-1} R$ can be done by accumulating constants that are independent of $\delta$ and $R$. In implementing this approach, we use a phase-space approach that alters the support of $f_{1,\alpha}, f_2,..,f_k$ by a factor 
 $\approx R^{-\frac12}$ where $R \geq \delta^{-2}$. This is fine with $f_2,..,f_k$
but not with $f_{1,\alpha}, \alpha \in J$ given that their support in some directions is $\mu \ll \delta$. 
This will require extra care. 

We work under the hypothesis that $U_i \subset B_i(0;\delta), i=2,..,k$, where $B_i(0;\delta)$ is the ball in the hyperplane $\calH_i$. For a function 
$f_i: \calH_i \rightarrow \C$ its margin is defined as follows
\begin{equation} \label{defmrg3}
\mbox{margin}^i(f_i) := \mbox{dist}(\mbox{supp} ( f), B_i(0;2\delta)^c), \quad i=2,..,k, 
\end{equation}
where $\mbox{supp} $ is the support of $f_i$.

We work under the hypothesis that $U_{1,\alpha} \subset B'(0;\delta) \times B''(0;\mu)$, where $B'(0;\delta)$ is the ball in the hyperplane 
$\tilde \calH_{1,\alpha}$ centered at the origin and of diameter $\delta$ and $B'(0;\mu)$ is the ball in the hyperplane $(\tilde \calH_{1,\alpha})^\perp$ centered at the origin and of diameter $\mu$. Accordingly, we split
the coordinates in $\calH_{1,\alpha}$ as follows $\xi^\alpha=(\xi^{',\alpha},\xi^{'',\alpha})$ where $\xi^{',\alpha}$ is the coordinate in $\tilde \calH_{1,\alpha}$ and $\xi^{'',\alpha}$ is the coordinate in $(\tilde \calH_{1,\alpha})^\perp$.
For a function 
$f: \calH_{1,\alpha} \rightarrow \R$ its margin is define by
\begin{equation} \label{defmrpg2}
\mbox{margin}^{1,\alpha}(f) := \inf_{\xi^{'',\alpha}} \mbox{dist}(\mbox{supp}_{\xi^{',\alpha}} (f(\cdot,\xi^{'',\alpha})), B'(0;2\delta)^c), 
\end{equation}
where $\mbox{supp}_{\xi^{',\alpha}}$ is the support of $f$ in the $\xi^{',\alpha}$ variable. On the physical side we denote by $x^{', \alpha}, x^{'', \alpha}$ the dual variables to $\xi^{', \alpha}, \xi^{'', \alpha}$, respectively. We complete the system of coordinates to 
$(\xi_1^\alpha, \xi^{', \alpha}, \xi^{'', \alpha})$ and $(x_1^\alpha,x^{', \alpha}, x^{'', \alpha})$, with $\xi_1^\alpha$ being the coordinate
in the direction of $N_{1,\alpha}$, the normal to $\calH_{1,\alpha}$ and $x_1^\alpha$ being the dual coordinate.

Our induction will aim at quantifying the behavior of $A(R)$ defined below. 

\begin{defin} Given $R \geq \delta^{-2}$ we define $A(R)$ to be the best constant for which the estimate
\begin{equation} \label{indh}
\| (\sum_{\alpha} | \calE_{1,\alpha}  f_{1,\alpha} |) \Pi_{i=2}^{k} \calE_i f_i \|_{L^{\frac2{k-1}}(Q)} \leq A(R) \Pi_{i=1}^{k} \| f_i \|_{L^2} 
\end{equation}
holds true for all parallelepipeds $Q \in \calC(R)$, with $f_i$
obeying the margin requirement
\begin{equation} \label{mrpg2}
margin^i(f_i) \geq \delta-R^{-\frac12}, i=2,..k, \qquad margin^{1,\alpha}(f_{1,\alpha}) \geq \delta-R^{-\frac12}, \forall \alpha \in J,
\end{equation}
and $f_{1,\alpha}$ is supported in $B(\tilde \calH_{1,\alpha};\mu) \subset \calH_{1,\alpha}, \forall \alpha \in J$.  
\end{defin}

Note that in \eqref{indh} we use absolute values. This indicates that we do not use any cancellation properties between the components 
$\calE_{1,\alpha}  f_{1,\alpha}$.
However, using the stronger statement with the absolute values plays a crucial role in carrying out the induction argument.

We start with the parallelepiped $Q$ of size $\delta^{-1} R$ centered at the origin. To keep notation compact we define
\[
H= \Pi_{i=2}^{k} \calE_i f_i , \quad G= \Pi_{i=2}^{k} \| f_i \|_{L^2}. 
\]
For each $q \in \calC(R) \cap Q$, the induction hypothesis is the following:
\begin{equation} \label{IO3}
\| (\sum_{\alpha} | \calE_{1,\alpha}  f_{1,\alpha} |) \cdot  H \|_{L^{\frac2{k-1}}(q)} 
\leq A(R)  \left( \sum_{\alpha} \| f_{1,\alpha} \|_{L^2(U_{1,\alpha})} \right) G. 
\end{equation}

We claim the following strengthening of \eqref{IO3}:
\begin{equation} \label{INS3}
\begin{split}
& \| (\sum_{\alpha} | \calE_{1,\alpha}  f_{1,\alpha} |) \cdot H \|_{L^{\frac2{k-1}}(q)} \\
& \les_N A(R) \Pi_{i=2}^{k} 
\left( \sum_{q' \in \calC\calH_i(R)} \la \frac{d(\pi_{N_i} q,q')}R \ra^{-(2N-n^2) } \| \la \frac{x-c(q')}R \ra^{N} \chi_{q'} \mathcal{F}_i^{-1} f_i \|_{L^2}^2 \right)^\frac12 \\
& \cdot  \sum_\alpha  \left( \sum_{\mathfrak{s}^\alpha \in \mathfrak{S}_{1,\alpha}(R)}  \la \frac{d(\tilde \pi_{1,\alpha}(q) , \mathfrak{s}^\alpha)}R \ra^{-(2N -2k)}  \| \la \frac{x^{',\alpha}- c(\mathfrak{s}^\alpha)}R \ra^N \chi_{\mathfrak{s}^\alpha}  \mathcal{F}_{1,\alpha}^{-1} f_{1,\alpha} \|^2_{L^2} \right)^\frac12.
\end{split}
\end{equation}
Similar improvements were provided in \cite{Be1}; in particular the improvement for the terms $f_i, i=2,..,k$ was established as claimed above 
(it can also be derived along similar, but simpler, lines as the arguments we provide below for the $f_{1,\alpha}$ terms). 
The improvement for $f_{1,\alpha}$ was also provided in \cite{Be1} in the case when there is also
one function $f_{1,\alpha}$, that is $J$ contains one element only. Here we provide an argument for general finite sets $J$ and note that the cardinality of $J$ does not impact $A(R)$. 

Therefore, in justifying \eqref{INS3} we focus on the improvement for the $f_{1,\alpha}$ terms only. 
Given $q \in \calC(R) \cap Q$ and $d \in \N$, let 
$A^\alpha(q,d)=\{ \mathfrak{s} \in \mathfrak{S}_{1,\alpha}(r): \la \frac{d(\tilde \pi_{1,\alpha}(q) , \mathfrak{s})}R \ra \approx d \}$. We can modify the sets such that each strip $\mathfrak{s}$ belongs to only one $A^\alpha(q,d)$.

From \eqref{com} we obtain the identity:
\begin{equation} \label{com3}
\begin{split}
& \sum_\alpha \sum_{\mathfrak{s}^\alpha \in A^\alpha(q,d)} |(x^{',\alpha}- c(\mathfrak{s}^\alpha)-x_1^\alpha \nabla_{\xi^{'\alpha}} \varphi_1(\frac{D^\alpha}i)) \calE_{1,\alpha}  \mathcal{F}_{1,\alpha}   \chi_{\mathfrak{s}^\alpha} \mathcal{F}_{1,\alpha}^{-1} f_{1,\alpha}| \\
= & \sum_\alpha \sum_{\mathfrak{s}^\alpha \in A^\alpha(q,d)} | \calE_1^\alpha \mathcal{F}_{1,\alpha}  (x^{',\alpha}- c(\mathfrak{s}^\alpha))   \chi_{\mathfrak{s}^\alpha} \mathcal{F}_{1,\alpha}^{-1} f_{1,\alpha}|,
\end{split}
\end{equation}
where the differential operator $\nabla_{\xi^{'\alpha}}  \varphi_1(\frac{D^\alpha}i)$ has symbol 
$\nabla_{\xi^{'\alpha}}  \varphi_1(\xi^\alpha)$. We have the following sequence of estimates
\[
\begin{split}
& \| \sum_\alpha \sum_{\mathfrak{s}^\alpha \in A^\alpha(q,d)} |(x^{',\alpha}- c(\mathfrak{s}^\alpha)-x_1^\alpha \nabla_{\xi^{'\alpha}} \varphi_1(\eta_{1,\alpha})) \calE_{1,\alpha}  \mathcal{F}_{1,\alpha}   \chi_{\mathfrak{s}^\alpha} 
 \mathcal{F}_{1,\alpha}^{-1} f_{1,\alpha}| \cdot H \|_{L^{\frac2{k-1}}(q)} \\
\leq & \| \sum_\alpha \sum_{\mathfrak{s}^\alpha \in A^\alpha(q,d)} |(x^{',\alpha}- c(\mathfrak{s}^\alpha)-x_1^\alpha \nabla_{\xi^{'\alpha}} \varphi_1(\xi^\alpha)) \calE_{1,\alpha}  \mathcal{F}_{1,\alpha}   \chi_{\mathfrak{s}^\alpha} 
 \mathcal{F}_{1,\alpha}^{-1} f_{1,\alpha}| \cdot H \|_{L^{\frac2{k-1}}(q)} \\
+ & \| x_1^\alpha ( \nabla_{\xi^{'\alpha}} \varphi_1(\eta_{1,\alpha}) - \nabla_{\xi^{'\alpha}} \varphi_1(\xi^\alpha)) 
\calE_{1,\alpha}  \mathcal{F}_{1,\alpha}   \chi_{\mathfrak{s}^\alpha} 
 \mathcal{F}_{1,\alpha}^{-1} f_{1,\alpha}
 \cdot H  \|_{L^{\frac2{k-1}}(q)}, \\
 \end{split}
 \]
 We invoke  \eqref{com3} and continue with
\[
\begin{split}
= & \|  \sum_\alpha \sum_{\mathfrak{s}^\alpha \in A^\alpha(q,d)} | \calE_{1,\alpha} \mathcal{F}_{1,\alpha}  (x^{',\alpha}- c(\mathfrak{s}^\alpha))   \chi_{\mathfrak{s}^\alpha} \mathcal{F}_{1,\alpha}^{-1} f_{1,\alpha}| \cdot H \|_{L^{\frac2{k-1}}(q)} \\
+ & \| \sum_\alpha \sum_{\mathfrak{s}^\alpha \in A^\alpha(q,d)} x_1^\alpha  \calE_{1,\alpha} \mathcal{F}_{1,\alpha} 
(  \nabla_{\xi^{'\alpha}} \varphi_1(\eta_{1,\alpha}) - \nabla_{\xi^{'\alpha}} \varphi_1(\xi^\alpha)) \chi_{\mathfrak{s}^\alpha} \mathcal{F}_{1,\alpha}^{-1} f_{1,\alpha}) \|_{L^\frac2{k-1}}. \\
 \end{split}
 \]
 We apply the induction hypothesis, and use that inside $Q$ we have $|x_1^\alpha| \les \delta^{-1}R, \forall \alpha \in J$, to further continue with
\[
\begin{split}
\leq & A(R) \left( \sum_\alpha \sum_{\mathfrak{s}^\alpha \in A^\alpha(q,d)} \|  (x^{',\alpha}- c(\mathfrak{s}^\alpha))   \chi_{\mathfrak{s}^\alpha} \mathcal{F}_{1,\alpha}^{-1} f_{1,\alpha} \|_{L^2}  \right) G  \\
+ & A(R) \delta^{-1} R \left(   \sum_\alpha \sum_{\mathfrak{s}^\alpha \in A^\alpha(q,d)}  \| ( \nabla_{\xi^{'\alpha}} \varphi_1(\eta_{1,\alpha}) - \nabla_{\xi^{'\alpha}} \varphi_1(\xi^\alpha)) \chi_{\mathfrak{s}^\alpha} \mathcal{F}_{1,\alpha}^{-1} f_{1,\alpha} \|_{L^2} \right) G  \\
\les & A(R) \left(  \sum_\alpha \sum_{\mathfrak{s}^\alpha \in A^\alpha(q,d)}   \|  (x^{',\alpha}- c(\mathfrak{s}^\alpha))   \chi_{\mathfrak{s}^\alpha} \mathcal{F}_{1,\alpha}^{-1} f_{1,\alpha} \|_{L^2} + R \| \chi_{\mathfrak{s}^\alpha} \mathcal{F}_{1,\alpha}^{-1} f_{1,\alpha} \|_{L^2} \right)  G \\
\les & R A(R)  \left(  \sum_\alpha \sum_{\mathfrak{s}^\alpha \in A^\alpha(q,d)}   \|  \la \frac{x^{',\alpha}- c(\mathfrak{s}^\alpha)}{R} \ra   \chi_{\mathfrak{s}^\alpha} \mathcal{F}_{1,\alpha}^{-1} f_{1,\alpha} \|_{L^2}  \right)  G.
\end{split}
\]
 Note that it is in the above use of the induction estimate for $\calE_{1,\alpha} \mathcal{F}_{1,\alpha}  (x^{',\alpha}- c(\mathfrak{s}^\alpha))   \chi_{\mathfrak{s}^\alpha} \mathcal{F}_{1,\alpha}^{-1} f_{1,\alpha}$  that we need to tolerate
the relaxed support of $f_{1,\alpha}$. The margin of $f_{1,\alpha}$ is $\geq \delta-(\delta^{-1} R)^{-\frac12}=\delta-\delta^\frac12 R^{-\frac12}$ and it  is affected by the convolution
$\mathcal{F}_{1,\alpha} ( (x^{',\alpha}- c(\mathfrak{s}^\alpha))   \chi_{\mathfrak{s}^\alpha}) $  by a factor of at most $C R^{-1}$ which is smaller than $\frac12 \delta^\frac12 R^{-\frac12}$, provided that $\delta$ is small relative to $C^{-1}$. Hence the new margin is $\geq \delta-\frac12 \delta^\frac12 R^{-\frac12} \geq \delta - R^{-\frac12}$, this being the required margin for using the induction hypothesis on cubes of size $R$.

We claim that for any $\mathfrak{s}^\alpha \in A^\alpha(q,d)$
\[
\| \la \frac{x^{',\alpha}- c(\mathfrak{s}^\alpha)-x_1^\alpha \nabla_{\xi^{',\alpha}} \varphi_1(\eta_{1,\alpha})}R \ra \|_{L^\infty(q)} 
\approx \la \frac{d(\tilde \pi_{1,\alpha}(q) , \mathfrak{s}^\alpha)}R \ra \approx d
\]
uniformly in $\alpha$. This statement is invariant to rotations of coordinates, therefore we can assume that $\nabla_{\xi^\alpha} \varphi_1(\eta_{1,\alpha})=0$ and moreover that $x^{',\alpha}=(x_2,..,x_k), x^{'',\alpha}=(x_{k+1},..,x_{n+1})$. This way, $\tilde \pi_{1,\alpha}(x)=(0,x_2,..,x_k,0,..,0)$ and the statement is
obvious. 

From the above we obtain that, for $d$ large,
\[
\begin{split}
& d R \sum_\alpha \sum_{\mathfrak{s}^\alpha \in A^\alpha(q,d)} | \calE_{1,\alpha} \mathcal{F}_{1,\alpha}   \chi_{\mathfrak{s}^\alpha} \mathcal{F}_{1,\alpha}^{-1} f_{1,\alpha} |  \\
 \les & \sum_\alpha \sum_{\mathfrak{s}^\alpha \in A^\alpha(q,d)} |(x^{',\alpha}- c(\mathfrak{s}^\alpha)-x_1^\alpha \nabla_{\xi^{'\alpha}} \varphi_1(\eta_{1,\alpha})) \calE_{1,\alpha}  \mathcal{F}_{1,\alpha}   \chi_{\mathfrak{s}^\alpha} 
 \mathcal{F}_{1,\alpha}^{-1} f_{1,\alpha}|.
\end{split}
\]
Combining all the above estimates gives
\[
\begin{split}
& d R \| \sum_\alpha \sum_{\mathfrak{s}^\alpha \in A^\alpha(q,d)} | \calE_{1,\alpha} \mathcal{F}_{1,\alpha}   \chi_{\mathfrak{s}^\alpha} \mathcal{F}_{1,\alpha}^{-1}  f_{1,\alpha} | \cdot  H \|_{L^{\frac2{k-1}}(q)} \\
\les & \| \sum_\alpha \sum_{\mathfrak{s}^\alpha \in A^\alpha(q,d)} |((x^{',\alpha})- c(\mathfrak{s}^\alpha)-x_1^\alpha \nabla_{\xi^{',\alpha}} \varphi_1(\eta_{1,\alpha})) \calE_{1,\alpha}  \mathcal{F}_{1,\alpha}  \chi_{\mathfrak{s}^\alpha} \mathcal{F}_{1,\alpha}^{-1} f_{1,\alpha}| \cdot H \|_{L^{\frac2{k-1}}(q)} \\
\les & R A(R)  \left(  \sum_\alpha \sum_{\mathfrak{s}^\alpha \in A^\alpha(q,d)}   \|  \la \frac{x^{',\alpha}- c(\mathfrak{s}^\alpha)}{R} \ra   \chi_{\mathfrak{s}^\alpha} \mathcal{F}_{1,\alpha}^{-1} f_{1,\alpha} \|_{L^2}  \right)  G.
\end{split}
\]
From this we conclude with (after more iterations of the same argument) 
\[
\begin{split}
& \| \sum_\alpha \sum_{\mathfrak{s}^\alpha \in A^\alpha(q,d)} |\calE_{1,\alpha} \mathcal{F}_{1,\alpha}   \chi_{\mathfrak{s}^\alpha} \mathcal{F}_{1,\alpha}^{-1}  f_{1,\alpha}| \cdot H\|_{L^{\frac2{k-1}}(q)} \\
\les & d^{-N} A(R)  \sum_\alpha \sum_{\mathfrak{s}^\alpha \in A^\alpha(q,d)} \|  \la \frac{x^{',\alpha}- c(\mathfrak{s}^\alpha)}R \ra^N \chi_{\mathfrak{s}^\alpha} \mathcal{F}_{1,\alpha}^{-1} f_{1,\alpha} \|_{L^2}  \cdot G. 
\end{split}
\]
Note, that while the argument above assumed $d$ large, this last inequality holds for all $d$'s, since it is trivial for $d$ small.
The summation over $d$ is done in the usual manner 
\[
\begin{split}
& \|  \sum_\alpha  |\calE_{1,\alpha} \ f_{1,\alpha}|  H \|_{L^{\frac2{k-1}}(q)}^\frac2{k-1} \\
= & \| \sum_d \sum_\alpha \sum_{\mathfrak{s}^\alpha \in A^\alpha(q,d)} 
|\calE_{1,\alpha} \mathcal{F}_1   \chi_{\mathfrak{s}^\alpha} \mathcal{F}_1^{-1}  f_{1,\alpha}|  H \|_{L^{\frac2{k-1}}(q)}^\frac2{k-1} \\
\les & \sum_d \| \sum_\alpha \sum_{\mathfrak{s}^\alpha \in A^\alpha(q,d)} | \calE_{1,\alpha} \mathcal{F}_{1,\alpha}   \chi_{\mathfrak{s}^\alpha} \mathcal{F}_{1,\alpha}^{-1}  f_{1,\alpha}|  H \|_{L^{\frac2{k-1}}(q)}^\frac2{k-1} \\
\les & (A(R))^\frac{2}{k-1}  \sum_d d^{-N \cdot \frac2{k-1} } \left( \sum_\alpha \sum_{\mathfrak{s}^\alpha \in A^\alpha(q,d)} \|  \la \frac{x^{',\alpha}- c(\mathfrak{s}^\alpha)}R \ra^N \chi_{\mathfrak{s}^\alpha}  \mathcal{F}_1^{-1} f_{1,\alpha} \|_{L^2}\right)^\frac2{k-1}   G^\frac2{k-1}. 
\end{split}
\]
Using \eqref{ab} together with the straightforward estimate
\[
\| d^{-\frac{k}2} \|_{l^\frac{2}{k-1}_\N} \les 1,
\]
and we can continue the sequence of inequalities we started above
\[
\begin{split}
& \|  \sum_\alpha  |\calE_{1,\alpha} \ f_{1,\alpha}| H \|_{L^{\frac2{k-1}}(q)}^\frac2{k-1}  \\
\les & (A(R))^\frac{2}{k-1}  \| d^{-(N -\frac{k}2)} \sum_\alpha \sum_{\mathfrak{s}^\alpha \in A^\alpha(q,d)} \|  \la \frac{x^{',\alpha}- c(\mathfrak{s}^\alpha)}R \ra^N \chi_{\mathfrak{s}^\alpha}  \mathcal{F}_1^{-1} f_{1,\alpha} \|_{L^2}  \|_{l^2_d}^\frac2{k-1}   G^\frac2{k-1} \\
\les & (A(R))^\frac{2}{k-1}  \left( \sum_\alpha  \| d^{-(N -\frac{k}2)}  \sum_{\mathfrak{s}^\alpha \in A^\alpha(q,d)} \| \la \frac{x^{',\alpha}- c(\mathfrak{s}^\alpha)}R \ra^N \chi_{\mathfrak{s}^\alpha}  \mathcal{F}_1^{-1} f_{1,\alpha} \|_{L^2} \|_{l^2_d} \right)^\frac2{k-1}   G^\frac2{k-1} \\
\les & (A(R))^\frac{2}{k-1}  \left( \sum_\alpha  \left( \sum_{\mathfrak{s}^\alpha}  \la \frac{d(\tilde \pi_{1,\alpha}(q) , \mathfrak{s}^\alpha)}R \ra^{-(2N -2k)}  \| \la \frac{x^{',\alpha}- c(\mathfrak{s}^\alpha)}R \ra^N \chi_{\mathfrak{s}^\alpha}  \mathcal{F}_1^{-1} f_{1,\alpha} \|^2_{L^2} \right)^\frac12 \right)^\frac2{k-1}  G^\frac2{k-1}.
\end{split}
\]
In passing to the last line we have used that the cardinality of $A^\alpha(q,d)$ is $\approx \la d \ra^{k-1}$ in order 
to bound the $l^1_{\mathfrak{s}^\alpha \in A^\alpha(q,d)}$ norm of the summand by the $l^2_{\mathfrak{s}^\alpha \in A^\alpha(q,d)}$ of the same quantity. 

We are done with the justification of \eqref{INS3} and continue with the final step in the induction on scales. 
 We define the functions $g_i:  \mathcal{L} (\calH_i) \rightarrow \R$ for $i=2,..,k$ by 
\[
g_i(\bj)= \left( \sum_{q' \in \calC\calH_{i}(R)} \la  \frac{d(q(\bj),q')}{R} \ra^{-(2N-n^2) } \| \la \frac{x'-c(q')}{R} \ra^N \chi_{q'} \mathcal{F}_i^{-1} f_i \|_{L^2}^2 \right)^\frac12, 
\]
for $\bj \in \mathcal{L}(\calH_i)$, while $g_1: \mathcal{L} \rightarrow \R$ by
\[
g_1(\bj)= \sum_\alpha  \left( \sum_{\mathfrak{s}^\alpha}  \la \frac{d(\tilde \pi_{1,\alpha}(q(\bj)) , \mathfrak{s}^\alpha)}R \ra^{-(2N -2k)}  \| \la \frac{x^{',\alpha}- c(\mathfrak{s}^\alpha)}R \ra^N \chi_{\mathfrak{s}^\alpha}  \mathcal{F}_1^{-1} f_{1,\alpha} \|^2_{L^2} \right)^\frac12 
\]
for $\bj \in \mathcal{L}$. Using \eqref{SN}, it is obvious that, provided $N$ is large enough (in terms of $n$ only), the following holds true:
\[
\| g_i \|_{l^2(\mathcal{L}(\calH_i)} \ls \| f_i \|_{L^2}, \quad i=2,..,k.
\] 
We also claim that
\begin{equation} \label{gst}
\| g_1 \|_{l^\infty_{z_1,z_{k+1},..,z_{n+1}} l^2_{z_2,..,z_k}(\mathcal{L})} \ls \sum_\alpha \| f_{1,\alpha} \|_{L^2}.
\end{equation}
This follows from the following geometrical observation: say $\bj=\sum_{i=1}^{n+1} z_i N_i$ where $z_i \in \Z$. We fix $z_1,z_{k+1},..,z_{n+1}$
and note that as we vary $z_2,..,z_k$, $\tilde \pi_{1,\alpha}(q(\bj))$ are almost disjoint and, most important, the strips they generate $\mathfrak{s}^\alpha (\tilde \pi_{1,\alpha} q(\bj)) \subset \mathfrak{S}_{1,\alpha}(R)$ are almost disjoint for each $\alpha \in J$ (given a point in $\calH_{1,\alpha}$ there are  finitely many $\bj$ such that the point belongs to $\mathfrak{s}^\alpha (\tilde \pi_{1,\alpha} q(\bj))$). This is due to the fact that the projections $\tilde \pi_{1,\alpha}$ onto the affine subspace 
$\tilde \calH_{1,\alpha}$ are taken along directions that are transversal to $N_2,..,N_k$
and the the infinite sides of the strips are in directions that are transversal to $N_2,..,N_k$. Using this geometric observations, \eqref{gst} follows from the equivalent of \eqref{SN} for strips. 

Then we apply \eqref{LWd2} to conclude with
\[
\| (\sum_{\alpha} | \calE_{1,\alpha}  f_{1,\alpha} |) \cdot H \|_{L^{\frac2{k-1}}(Q)}  \ls A(R) \Pi_{i=1}^{k} \| f_i \|_{L^2}.
\]
Thus we obtain
\[
A(\delta^{-1} R) \leq C A(R)
\]
for a constant $C$ that is independent of $\delta$ and $R$. Iterating this gives $A(\delta^{-N} r) \leq C^{N} A(r)$.
Therefore $\max_{r \in [0,\delta^{-2}]} A(\delta^{-N} r) \leq C^N \max_{r \in [0,\delta^{-2}]} A(r)= C^N C(\delta) \mu^{\frac{n+1-k}2}$ is
obtained from the uniform pointwise bound
\begin{equation} \label{final}
\begin{split}
\| (\sum_{\alpha} | \calE_{1,\alpha}  f_{1,\alpha} |) \Pi_{i=2}^{k} \calE_i f_i \|_{L^\infty} 
& \ls \| \sum_{\alpha} | \calE_{1,\alpha}  f_{1,\alpha} | \|_{L^\infty} \Pi_{i=2}^{k}  \| \calE_i f_i \|_{L^\infty}  \\
& \ls \mu^{\frac{n+1-k}2} \left( \sum_{\alpha} \| \calE_{1,\alpha}  f_{1,\alpha} \right)  \|_{L^2} \Pi_{i=2}^{k} \| f_i \|_{L^2}
\end{split}
\end{equation}
which is integrated over arbitrary cubes of size $\leq \delta^{-2}$. Note that we have used the support properties of $f_{1,\alpha}$
to obtain the improved bound. 

For  $R \in [\delta^{-N},\delta^{-N-1}]$, the above implies
\[
A(R) \leq C^N C(\delta) \mu^{\frac{n+1-k}2} \leq R^\epsilon C(\delta) \mu^{\frac{n+1-k}2}
\]
provided that $C^N \leq \delta^{-N\epsilon}$. Therefore choosing $\delta=C^{-\frac{1}{\epsilon}}$ leads to the desired result. 

\end{proof}

\begin{proof}[Proof of Corollary \ref{MBcor}]  In each $\calH_i, i=1,..,k$, $y_i \in \R$, we define $\calH_i + y_i N_i$ to be the translate of $\calH_i$ by $y_i N_i$. Also $\calC \calH_i(r) + y_i N_i$ is the corresponding translate of $\calC \calH_i(r)$ by $y_i N_i$. 

Given any vector $y \in \R^{n+1}$ with $|y_i- c_i(q)| \leq r, i=1,..,k$ and $y_i=c_i(q), k+1 \leq i \leq n+1$, we claim the following inequality:
\[
\begin{split}
& \| (\sum_{\alpha} | \calE_{1,\alpha}  f_{1,\alpha} |) \cdot H \|_{L^{\frac2{k-1}}(q)} \\
& \les_N C(\epsilon) r^\epsilon \mu^{\frac{n-k+1}2}  \Pi_{i=2}^{k} 
\left( \sum_{q' \in \calC\calH_i(r) +y_i N_i} \la \frac{d(\pi_{N_i} q,q')}r \ra^{-(2N-n^2) } \| \la \frac{x-c(q')}r \ra^{N} \chi_{q'} \mathcal{F}_i^{-1} f_i \|_{L^2(\calH_i+y_i N_i)}^2 \right)^\frac12 \\
& \cdot  \sum_\alpha  \left( \sum_{q' \in \calC_{1,\alpha}(r)+y_1 N_1}  \la \frac{d(\tilde \pi_{1,\alpha}(q) , q')}r \ra^{-(2N -2k)}  \| \la \frac{x^{\alpha}- c(q')}r \ra^N \chi_{q'}  \mathcal{F}_{1,\alpha}^{-1} f_{1,\alpha} \|^2_{L^2(\calH_{1,\alpha}+y_1 N_1)} \right)^\frac12.
\end{split}
\]
It suffices to prove this estimate for $y=0$, in which case it is very similar to \eqref{INS3}. Except that, for the $f_{1,\alpha}$
terms we do not use strips, but cubes. This should be a reason for concern, as the use of strips was necessary to keep the localization
of the $f_{1,\alpha}$ at scale $\mu$ intact throughout the induction process. However, given that $\mu \approx r^{-1}$, the multiplication with 
$\chi_{q'}$ alters the localization by a factor of $r^{-1} \approx \mu$. A similar argument to the one used in the proof of \eqref{INS3}
gives the above estimate. 

Next we average the above estimate  over the values of $(y_1,..,y_k)$ 
satisfying $|y_i- c_i(q)| \leq r$ (keeping $y_i=c_i(q), i \geq k+1$) to obtain
\[
\begin{split}
& \| \Pi_{i=1}^{k} \calE_i f_i \|_{L^1(q)} \les C(\epsilon) r^\epsilon (r^{-1})^{\frac{n-k+1}2} r^{-\frac{k}2} \\
  &  \cdot \Pi_{i=1}^{k}  \left( \int_{|y_i| \leq r} \sum_{q' \in \calC \calH_i(r)+y_i N_i} \la \frac{d(\pi_{N_i} q,q')}r \ra^{-N} \| \la \frac{x-c(q')}r \ra^{N} \chi_{q'} \calE_i f_i \|_{L^2(\calH_i+y_i N_i)}^2 \right)^\frac12 \\
  & \les  C(\epsilon) r^\epsilon r^{-\frac{n+1}2} \Pi_{i=1}^{k} \| \tilde \chi_q \calE_i f_i \|_{L^2}
\end{split}
\]
This finishes the proof. 

\end{proof}

\section{A new multilinear estimate} \label{CMS}

In this Section we address \eqref{CK}, the last supporting detail in the proof of Proposition \ref{NLP}.
As described in Section \ref{SePa}, we are given $k$ smooth hypersurfaces $S_i = \Sigma_i(U_i)$ with smooth parameterizations 
$\Sigma_i$ obeying \eqref{normal2}. These hypersurfaces can be thought as living in the frequency space and generate the operators $\calE_i$. 
In addition we are given another smooth surface $S$ of dimension $n-k+1$, that should be thought as living in the physical space,
with the following properties:

\bf P1: \rm $S$ is uniformly transversal to $N_1(\zeta_1), .., N_{k}(\zeta_{k})$ for all choices $\zeta_i \in S_i$: there exists $\nu >0$
such that, for any $\zeta_i \in S_i, i=1,..,k$, for any $y \in S$ and for any orthonormal basis $v_{k+1},..,v_{n+1}$ of $T_y S$, the following holds true 
\[
vol (N_1(\zeta_1), .., N_{k}(\zeta_{k}), v_{k+1},..,v_{n+1}) \geq \nu.
\]

\bf P2: \rm There exists $\nu >0$ such that for any $P_1,P_2 \in S$, for any $\zeta_1 \in S_1$, for any $\zeta_i \in S_i, \zeta_j \in S_j, 2 \leq i < j \leq k$ and for any $\alpha_i,\alpha_j \in \R$, the following holds true
\begin{equation} \label{trans3}
vol(\overrightarrow{P_1 P_2}, N_1(\zeta_1), \overrightarrow{v}) \geq \nu |\overrightarrow{P_1 P_2}| \cdot |\overrightarrow{v}|,
\end{equation}
where $ \overrightarrow{v}= \alpha_i N_i(\zeta_i) - \alpha_j N_j(\zeta_j)$.

As already mentioned in Section \ref{SePa}, in this section we make no curvature assumptions on $S_i$. 
However, we note that property \bf P2 \rm follows from curvature properties similar to those used  in the main Theorem \ref{mainT};
in other words the curvature properties have been encoded in the structure of $S$.

Given $r>0$, we recall that $\calC(r)$ is the set of unit cubes in $\R^{n+1}$ with centers in the lattice $r \Z^{n+1}$.
With $S$ as above and for each $q \in \calC(r)$, we define
\[
S(q)= q+ S \cap B(0, r^2). 
\]
Here $S \cap B(0, r^2)$ should be understood as follows: we cut the surface $S$ at scale $\approx r^2$, and whether this is performed in a ball or cube, centered at the origin or somewhere else, it is unimportant. The reason for doing this comes from the use of wave packets and their scales.

More generally, given a subset $A \subset \calC(r)$, we define
\[
S(A)= \cup_{q \in A} S(q). 
\]

The main result of this section is the following:
\begin{theo} \label{NewRE} Assume that $S_i, i=1,..,k$ and $S$ are as above. Then for any $\epsilon > 0$, there is $C(\epsilon)$ such that the following holds true
\begin{equation} \label{CK1}
\left( \sum_{q \in \calC(r) \cap B(0,r^2)} \left( \| \calE_1 f_1 \|_{L^2(S(q))} \Pi_{i=2}^{k} \| \calE_i f_i \|_{L^2(q)} \right)^\frac{2}{k-1}  \right)^\frac{k-1}2
\leq C(\epsilon) r^\frac{k}2 r^\epsilon \Pi_{i=1}^{k} \| f_i \|_{L^2(U_i)}.
\end{equation} 

\end{theo}

The above result has a multilinear flavor to it. The factor $r^{\frac{k}2}$ has to appear because we consider the mass of 
$\calE_i f_i$ in neighborhoods of size $r$ of hypersurfaces across which we would have good energy estimates, see the proof of the Theorem for details. Otherwise \eqref{CK1} is similar to a multilinear restriction estimate, see \eqref{AMRE} (with $\mu = 1$), except that now, one of the objects, $\calE_1 f_1$ is measured in a more complex fashion. 

The complexity is not so much from the fact that we collect energy from various spatial regions; indeed if
$v_i$ arbitrary vectors of any length, then and estimate of type
\[
\| \| \Pi_{i=1}^k\| \calE_i f_i \|_{L^2(q+v_i)} \|_{l^{\frac2{k-1}}_q} 
\les r^{\frac{k}2} r^\epsilon \Pi_{i=1}^k \| f_i \|_{L^2(U_i)}
\]
is similar to the one with $v_i=0$ which in turn is similar to \eqref{AMRE} (with $\mu = 1$). 

The complexity has to do with the factor $ \| \| \calE_1 f_1 \|_{L^2(U_i)} \|_{L^2(q+v_1)} $ being replaced with $ \| \calE_1 f_1 \|_{L^2(q+S)}$, that is with collecting  the energy of $\calE_1 f_1$ not only across a cube $q+v_1$, but across a thickened surface $q+S$. It is the dimensionality of the surface $S$ being $n-k+1$ versus that of $v_1$ being $0$ that changes the character of the estimate. Another feature to point out is the following: the classical mutilinear estimate improves under certain localization properties of the support of the interacting functions (see the $\mu$ factor in \eqref{AMRE});
 \eqref{CK1} does not improve under such localizations. 

In \cite{Be2} we provided an energy estimate of the following type
\begin{equation} \label{EE}
\| \calE_1 f_1 \|_{L^2(\tilde S+q)} \les r^\frac12 \| f_1 \|_{L^2(U_1)}
\end{equation}
where $\tilde S$ is a hypersurface (i.e., of codimension $1$) that is transversal to the propagation directions of $\calE_1 f_1$,
that is to any $N_1(\zeta_1)$ with $\zeta_1 \in S_1$.  

The starting point of the arguments in this section is a refinement of \eqref{EE} in terms of wave packets. We use the result of Lemma \ref{LeWP}
with $c=1$ and $R=4r^2$ to obtain the wave packet decomposition
\[
\calE_1 f_1 = \sum_{T_1 \in \calT_1} \phi_{T_1}.
\]
We also recall the definition of $c_N(T_1)$ from \eqref{qest1}
and their property \eqref{massc}.

\begin{lema} \label{tub}
There exists $N \in \N$ such that  for any $q \in \calC(r)$ centered inside $B(0,r^2)$, the following holds true:
\begin{equation} \label{EfT}
\| \calE_1 f_1 \|_{L^2(S(q))} \les r^\frac12  \left( \sum_{T_1 \in \calT_1} \la \frac{d(T_1,S(q))}r \ra^{-N}  c_{2N}(T_1)^2) \right)^\frac12. 
\end{equation}
\end{lema}
\begin{proof} We start by noting that since $S(q) \subset B(0,4r^2), \forall q \in \calC(r)$ centered inside $B(0,r^2)$. 

We write
\[
\begin{split}
\| \calE_1 f_1 \|_{L^2(S(q))}^2 & \les \sum_{q' \cap S(q) \ne \emptyset}  \| \calE_1 f_1 \|_{L^2(q')}^2 \\
& \les \sum_{q' \cap S(q) \ne \emptyset}  \sum_{T_1 \in \calT_1} \| \phi_{T_1} \|_{L^2(q')}^2 \\
& =  \sum_{T_1 \in \calT_1}  \sum_{q' \cap S(q) \ne \emptyset} \| \phi_{T_1} \|_{L^2(q')}^2 \\
& \les  \sum_{T_1 \in \calT_1}  \sum_{q' \cap S(q) \ne \emptyset} \tilde \chi_{T_1}(x_{q'},t_{q'})^{N} \tilde \chi_{T_1}(x_{q'},t_{q'})^{-N} \| \phi_{T_1} \|_{L^2(q')}^2 \\
& \les \sum_{T_1 \in \calT_1} \la \frac{d(T_1,S(q))}r \ra^{-N}  r c_{2N}(T_1)^2. \\
\end{split}
\]
In justifying the last line we used the following two estimates: the obvious estimate 
\[
supp_{q' \cap S(q) \ne \emptyset} \tilde \chi_{T_1}(x_{q'},t_{q'})^{N} \les \la \frac{d(T_1,S(q))}r \ra^{-N},
\]
as well as
\begin{equation} \label{t1}
 \sum_{q' \cap S(q) \ne \emptyset} \tilde \chi_{T_1}(x_{q'},t_{q'})^{-N} \| \phi_{T_1} \|_{L^2(q')}^2 \les r c_{2N}(T_1)^2. 
\end{equation}
We justify \eqref{t1} as follows: from \eqref{qest1} we obtain
\[
 supp_{q'} \tilde \chi_T(x_{q'},t_{q'})^{-2N} \| \phi_{T_1} \|_{L^2(q')}^2 \les r c_{2N}(T_1)^2.  
\]
Then \eqref{t1} follows from
\[
\sum_{q' \cap S(q) \ne \emptyset} \tilde \chi_{T_1}(x_{q'},t_{q'})^{N} \les 1. 
\]
But, choosing $N$ large enough, this is a direct consequence of the transversality between $T_1$ and $S(q)$. 

\end{proof}

\begin{proof}[Proof of Theorem \ref{NewRE}]  As we already explained in the proof of Theorem \ref{MB2}, it suffices to establish the result under the following assumption: given some $0 < \delta \ll 1$, the diameter of $U_i$ is $ \leq \delta$. 

The setup is also similar to the one in Section \ref{LocME}. We pick $\zeta_i^0 \in \Sigma_i$, let $N_i=N_i(\zeta_i^0)$ be the 
normal to $\Sigma_i$ and let $\calH_i$ be the transversal hyperplane passing through the origin with normal $N_i(\zeta_i^0)$. Using a smooth change of coordinates, we can assume that $U_i \subset \calH_i$ and that 
\begin{equation} \label{E}
\calE_i f_i = \int_{U_i} e^{i (x' \xi' + x_i \varphi_i(\xi'))} f_i(\xi') d\xi',
\end{equation}
where $x=(x_i,x')$, $x_i$ is the coordinate in the direction of $N_i$ and $x'$ are the coordinates in the directions from $\calH_i$. 
Since the diameter of $U_i$ is $\les \delta$,  it follows that $|\nabla \varphi_i(x)-\nabla \varphi_i(y)| \les \delta$ for any $x,y \in U_i$.  
Using the normals $N_i$ we construct all entities described in Section \ref{NOT} as well as the margin of a function $f: \calH_i \rightarrow \C$ as defined in \eqref{defmrg3}.

We complete the system of vectors by choosing $N_{k+1},..,N_{n+1}$ such that \eqref{normal} is satisfied.
We then construct the lattice $\mathcal{L}:=\{ z_1 N_1 + ... +z_{n+1} N_{n+1}: (z_1,..,z_{n+1}) \in \Z^{n+1} \}$ and 
for given $r > 0$ we let $\calC(r)$ be the set of of parallelepipeds of size $r$ in $\R^{n+1}$ relative to the lattice $\mathcal{L}$.
The lattice $\mathcal{L}$ and the set of parallelepipeds $\calC(r)$ obtained this way are "oblique", thus different than the one claimed in \eqref{CK1}, 
which is built on the standard orthonormal basis. However, passing from results in terms of an oblique lattice to the ones in the standard basis is easy:
it can be done by changing coordinates, or by direct estimates. 

Our induction will aim at quantifying the behavior of $A(R)$ defined below. 

\begin{defin} Given $r \leq R \les r^2 $ we define $A(R)$ to be the best constant for which the estimate
\begin{equation} \label{IH}
\begin{split}
& \left( \sum_{q \in \calC(r) \cap Q} \left( \| \calE_1 f_1 \|_{L^2(S(q))}
 \Pi_{i=2}^{k} \| \calE_i f_i \|_{L^2(q)} \right)^\frac{2}{k-1}  \right)^\frac{k-1}2 
\\
\leq & A(R) r^\frac{k}2  \left( \sum_{T_1 \in \calT_1} \la \frac{d(T_1,S(Q))}R \ra^{-N}  c_{2N}(T_1)^2 \right)^\frac12
 \Pi_{i=2}^{k} \| f_i \|_{L^2(U_i)}
\end{split}
\end{equation}
holds true for all parallelepipeds $Q \in \calC(R)$,and all $f_i \in L^2(U_i), i=2,..,k$
obeying the margin requirement
\begin{equation} \label{mrpg3}
margin^i(f_i) \geq \delta-R^{-\frac12}.
\end{equation} 
\end{defin}

The above estimate holds true for $R=r$ with $A(r) \approx 1$; indeed, it follows from \eqref{EfT} and the obvious estimate
\[
\| \calE_i f_i \|_{L^2(q)} \ls r^{\frac12} \| f_i \|_{L^2(U_i)}.
\]
Note also that we limit the range of the argument to $R \les r^2$. This is important so as to be able to use the wave packet described above. 

Next, we proceed with the induction step. We provide an estimate inside any cube $\textbf{Q} \in \calC(\delta^{-1}R)$ based on prior information on estimates inside cubes $Q \in \calC(R) \cap \textbf{Q}$. Without restricting the generality of the argument, we assume
that $\textbf{Q}$ is centered at the origin and recall that each $Q \in \calC(R) \cap \textbf{Q}$ has its center in
$ \mathcal{L}$. When such a  $Q$ is projected using $\pi_{N_i}$ onto $\calH_i$ one obtains $\pi_{N_i} Q \in \calC \calH_i$.
We let $Q_0$ be the cube in $\calC(R)$ centered at the origin. 

We strengthen the induction hypothesis \eqref{IH} to
\begin{equation} \label{INS}
\begin{split}
& \left( \sum_{q \in \calC(r) \cap Q} \left( \| \calE_1 f_1 \|_{L^2(S(q))}  \Pi_{i=2}^{k} \| \calE_i f_i \|_{L^2(q)} \right)^\frac{2}{k-1}  \right)^\frac{k-1}2 \\
\les  & A(R) \left( \sum_{T_1 \in \calT_1} \la \frac{d(T_1,S(Q))}R \ra^{-N}  c_{2N}(T_1)^2 \right)^\frac12 \\
& \cdot \Pi_{i=2}^{k} \left( \sum_{Q' \in \calC \calH_i(R)} \la \frac{d(\pi_{N_i} Q,Q')}R \ra^{-(N-2n^2)} \| \la \frac{x-c(Q')}R \ra^{N} \chi_{Q'} \mathcal{F}_i^{-1} f_i \|_{L^2}^2 \right)^\frac12
\end{split}
\end{equation}

The improvement for the terms $\calE_i f_i$ with $i \geq 2$ is standard by now, see \eqref{INS3} and the references to \cite{Be1}.  
Using \eqref{INS} we conclude the argument using the discrete Loomis-Whitney inequality in \eqref{LWd2}. 
For $i=2,..,n$, we define the functions $g_i: \mathcal{L}(\calH_i) \rightarrow \R$ by
\[
g_i(\bj)= \left( \sum_{Q' \in \calC \calH_{i}(R)} \la  \frac{d(Q(\bj),Q')}{R} \ra^{-(N-2n^2)} \| \la \frac{x'-c(Q')}{R} \ra^N \chi_{q'} \mathcal{F}_i^{-1} f_i \|^2_{L^2} \right)^\frac{1}2, \bj \in \mathcal{L}(\calH_i) .
\]
where we recall that $Q(\bj) \in \mathcal{CH}_i(R)$ is the cube centered at $R \bj$.

From \eqref{SN}, it is easy to see that for $N$ large enough (depending only on $n$),
$g_i \in l^2(\Z^n), i=2,..,k$ with
\[
\| g_i \|_{l^2(\mathcal{L}(\calH_i))} \ls \| f_i \|_{L^2}.
\] 
For $i=1$ and $\bj \in \mathcal{L} $,  we recall that $Q(\bj)= Q_0 + R \bj \in \calC(R)$ is the cube centered at $R \bj$, 
and define
\[
g_1(\bj)= \left( \sum_{T_1 \in \calT_1} \la \frac{d(T_1,S(Q(\bj)))}R \ra^{-N}  c_{2N}(T_1)^2 \right)^\frac12. 
\]
We claim that $g_1\in  l^\infty_{j_1,j_{k+1},..,j_{n+1}} l^2_{j_2,..,j_k}(D)$, where 
$D=\{\bj \in \mathcal{L}: \| \bj \|_{l^\infty} \leq  \delta^{-1} \}$ is the domain of interest, together with the estimate
\begin{equation} \label{finale}
\| g_1 \|_{l^\infty_{j_1,j_{k+1},..,j_{n+1}} l^2_{j_2,..,j_k}(D)} 
\ls \left( \sum_{T_1 \in \calT_1} \la \frac{d(T_1,S(\textbf{Q}))}{ \delta^{-1} R} \ra^{-N}  c_{2N}(T_1)^2 \right)^\frac12.
\end{equation}
We assume for a moment \eqref{finale} to be true. Using \eqref{INS}, we invoke \eqref{LWd2} and the above estimates
on $g_i$ to obtain
\[
\begin{split}
&  \left( \sum_{q \in \calC(r) \cap \textbf{Q}} \left( \| \calE_1 f_1 \|_{L^2(S(q))}  \Pi_{i=2}^{k} \| \calE_i f_i \|_{L^2(q)} \right)^\frac{2}{k-1} 
\right)^\frac{k-1}2  \\
= & \left( \sum_{Q \in \calC(R) \cap \textbf{Q} } \sum_{q \in \calC(r) \cap Q} \left( \| \calE_1 f_1 \|_{L^2(S(q))}  \Pi_{i=2}^{k} \| \calE_i f_i \|_{L^2(q)} \right)^\frac{2}{k-1} \right)^\frac{k-1}2 \\
\les & A(R) r^\frac{k}2  \left( \sum_{T_1 \in \calT_1} \la \frac{d(T_1,S(\textbf{Q}))}R \ra^{-N}  c_{2N}(T_1)^2 \right)^\frac12
 \Pi_{i=2}^{k} \| f_i \|_{L^2(U_i)}.
\end{split}
\]
Thus we establish that
\[
A(\delta^{-1} R) \les A(R). 
\]
This implies \eqref{CK1} in a standard fashion, see for instance \cite{Be1}, and concludes our proof.

We owe an argument for the claim \eqref{finale}. We fix $j_1,j_{k+1},..,j_{n+1}$ with $\max \{ |j_1|, |j_{k+1}|,..,| j_{n+1}|\} \leq \delta^{-1}$. 
Then  \eqref{finale} is a consequence of the estimate
\[
\sum_{j_2,..,j_k: |j_l| \leq \delta^{-1}} \sum_{T_1 \in \calT_1} \la \frac{d(T_1,S(Q(\bj)))}R \ra^{-N}  c_{2N}(T_1)^2
\les  \sum_{T_1 \in \calT_1} \la \frac{d(T_1,S(\delta^{-1} Q_0))}{\delta^{-1} R} \ra^{-N}  c_{2N}(T_1)^2,
\]
which in turn follows from the estimate
\begin{equation} \label{lcl}
\sum_{j_2,..,j_k: |j_l| \leq \delta^{-1}}  \la \frac{d(T_1,S(Q(\bj)))}{ R} \ra^{-N}   \les   \la \frac{d(T_1,S(\delta^{-1} Q_0))}{ \delta^{-1} R}  \ra^{-N}.
\end{equation}
\eqref{lcl} is easily derived from the following claim: given any $d \in \N$, there are $\les d^{k-1}$ values of $\bj \in D$ such that 
$d(T_1,S(Q(\bj))) \leq dR$. 

Thus, the last thing we need to do is establishing the claim above. Let $\bj_1, \bj_2 \in D$ be such that $d(T_1,S(Q(\bj_1))), d(T_1,S(Q(\bj_2))) \approx dR$. Let $L_1$ be the center line of $T_1$; it has direction $N_1=N_1(\zeta_1)$ for some $\zeta_1 \in S_1$. Using the fact that $R \geq r$, 
we conclude that there are points $P_1,P_2 \in T_1, \tilde P_1 \in S(Q(\bj_1)), \tilde P_2 \in S(Q(\bj_2))$ with the following properties:

- $P_1, P_2 \in L_1$ 

- $\tilde P_1 \in S + R \bj_1, \tilde P_2 \in S + R \bj_2$

- $d(P_1, \tilde P_1), d(P_2, \tilde P_2) \les dR$.

From the vector identity
\[
\overrightarrow{\tilde P_1 \tilde P_2} = \overrightarrow{\tilde P_1 P_1} + \overrightarrow{P_1 P_2} + \overrightarrow{P_2  \tilde P_2},
\]
and the above properties, we obtain 
\[
| \overrightarrow{\tilde P_1 \tilde P_2} - \overrightarrow{P_1 P_2}| \les dR. 
\]
On the other hand, $\tilde P_1 = Q_1 + R \bj_1, \tilde P_2 = Q_2 + R \bj_2$ for some $Q_1,Q_2 \in S$, therefore
\[
\overrightarrow{\tilde P_1 \tilde P_2} - \overrightarrow{P_1 P_2} = \overrightarrow{Q_1 Q_2}  + R(\bj_1 -\bj_2) + \alpha N_1, 
\]
for some $\alpha \in \R$. Now we bring in the transversality considerations, see \eqref{trans3}, to conclude that 
\[
d R \ges |\overrightarrow{Q_1 Q_2}  + R(\bj_1 -\bj_2) + \alpha N_1| \ges R |\bj_1-\bj_2|;
\]
here we use the structure of the lattice $\mathcal{L}$ to infer that $\bj_1-\bj_2 = \alpha_i N_i(\zeta_i) - \alpha_j N_j(\zeta_j)$ for some
$i,j \in \{2,..,k\}$ and some $\alpha_i, \alpha_j \in \R$. 

Thus $d \ges |\bj_1 - \bj_2|$, and, as a consequence, there are about $d^{k-1}$ values of $\bj$ with the property that $d(T_1,S(Q(\bj))) \leq dR$.

\end{proof}

\bibliographystyle{amsplain} \bibliography{HA-refs}

\end{document}